\documentclass[reqno,a4paper,11pt]{amsart}
\usepackage[all,poly]{xy}
\usepackage{amsfonts}
\usepackage{amssymb}
\usepackage{amsmath}
\usepackage{color}
\usepackage[pagebackref,colorlinks]{hyperref}
\usepackage{enumerate}

\DeclareFontFamily{U}{mathx}{\hyphenchar\font45}
\DeclareFontShape{U}{mathx}{m}{n}{
      <5> <6> <7> <8> <9> <10>
      <10.95> <12> <14.4> <17.28> <20.74> <24.88>
      mathx10
      }{}
\DeclareSymbolFont{mathx}{U}{mathx}{m}{n}
\DeclareFontSubstitution{U}{mathx}{m}{n}
\DeclareMathSymbol{\bigplus}        {1}{mathx}{"90}
\DeclareMathSymbol{\bigtimes}       {1}{mathx}{"91}
     
\theoremstyle{plain}
\newtheorem {lemma}{Lemma}

\newtheorem {theorem}[lemma]{Theorem}

\newtheorem {keylemma}[lemma]{Key Lemma}

\theoremstyle{definition}
\newtheorem{definition}[lemma]{Definition}

\newtheorem{remark}[lemma]{Remark}
\newtheorem {example}[lemma]{Example}

\newcommand{\R}{\mathbb{R}}
\newcommand{\N}{\mathbb{N}}
\newcommand{\Z}{\mathbb{Z}}
\newcommand{\X}{\langle X \rangle}
\def\Oo{\mathcal{O}}
\newcommand{\GKdim}{\operatorname{GKdim}}
\newcommand{\nod}{\operatorname{nod}}
\newcommand{\fin}{\operatorname{fin}}

\newcommand{\id}{\operatorname{id}}

\title[WLpas that are isomorphic to Lpas]{Weighted Leavitt path algebras that are isomorphic to unweighted Leavitt path algebras}

\author{Raimund Preusser}
\email{raimund.preusser@gmx.de}
\date{}


\begin{document}

\begin{abstract} 
Let $K$ be a field. We characterise the row-finite weighted graphs $(E,w)$ such that the weighted Leavitt path algebra $L_K(E,w)$ is isomorphic to an unweighted Leavitt path algebra. Moreover, we prove that if $L_K(E,w)$ is locally finite, or Noetherian, or Artinian, or von Neumann regular, or has finite Gelfand-Kirillov dimension, then $L_K(E,w)$ is isomorphic to an unweighted Leavitt path algebra.
\end{abstract}

\maketitle

\section{Introduction}
In a series of papers \cite{vitt56, vitt57, vitt62, vitt65} William Leavitt studied algebras that are now denoted by $L_K(n,n+k)$ and have been coined Leavitt algebras. Let $X=(x_{ij})$ and $Y=(y_{ji})$ be $(n+k)\times n$ and $n\times (n+k)$ matrices consisting of symbols $x_{ij}$ and $y_{ji}$, respectively. Then for a field $K$, $L_K(n,n+k)$ is the unital $K$-algebra generated by all $x_{ij}$ and $y_{ji}$ subject to the relations $XY=I_{n+k}$ and $YX=I_n$. The algebra $L_K(n,n+k)$ can be described as the $K$-algebra $A$ with a universal left $A$-module isomorphism $A^n\rightarrow A^{n+k}$, cf. \cite[second paragraph on p. 35]{bergman74}.

(Unweighted) Leavitt path algebras are algebras associated to directed graphs. They were introduced by G. Abrams and G. Aranda Pino in 2005 \cite{aap05} and independently by P. Ara, M. Moreno and E. Pardo in 2007 \cite{Ara_Moreno_Pardo}. For the directed graph
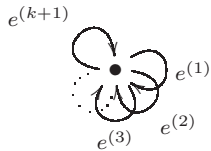
\[
\xymatrix{
& \bullet\ar@{.}@(l,d) \ar@(ur,dr)^{e^{(1)}} \ar@(r,d)^{e^{(2)}} \ar@(dr,dl)^{e^{(3)}} \ar@(l,u)^{e^{(k+1)}}& 
}
\]
with one vertex and $k+1$ loops one recovers the Leavitt algebra $L_K(1,k+1)$. The definition and the development of the theory were inspired on the one hand by Leavitt's construction of $L_K(1,k+1)$ and on the other hand by the Cuntz algebras $\Oo_n$ \cite{cuntz1} and the Cuntz-Krieger algebras in $C^*$-algebra theory \cite{raeburn}. The Cuntz algebras and later Cuntz-Krieger type $C^*$-algebras revolutionised $C^*$-theory, leading ultimately to the astounding
Kirchberg-Phillips classification theorem~\cite{phillips}. The Leavitt path algebras have created the same type of stir in the algebraic community.
 
There have been several attempts to introduce a generalisation of the Leavitt path algebras which would cover the algebras $L_K(n,n+k),~n\geq 2$ as well. In 2013, R. Hazrat \cite{hazrat13} introduced weighted Leavitt path algebras. These are algebras associated to weighted graphs. For the weighted graph 
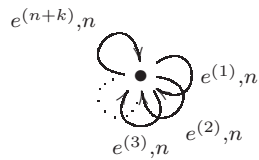
\[
\xymatrix{
& \bullet \ar@{.}@(l,d) \ar@(ur,dr)^{e^{(1)},n} \ar@(r,d)^{e^{(2)},n} \ar@(dr,dl)^{e^{(3)},n} \ar@(l,u)^{e^{(n+k)},n}& 
}
\]
with one vertex and $n+k$ loops of weight $n$ one recovers the Leavitt algebra $L_K(n,n+k)$. If the weights of all the edges are $1$, then the weighted Leavitt path algebras reduce to the unweighted Leavitt path algebras. 

Which are the new examples in the class of weighted Leavitt path algebras? In \cite{hazrat-preusser} it was shown that any simple or graded simple weighted Leavitt path algebra is isomorphic to an unweighted Leavitt path algebra. In \cite{preusser} and \cite{preusser1} it was shown that any finite-dimensional or Noetherian weighted Leavitt path algebra is isomorphic to an unweighted Leavitt path algebra. Furthermore, graph-theoretic criterions that are sufficient and necessary for $L_K(E,w)$ being finite-dimensional/Noetherian were found (see \cite[Theorems 25 and 52]{preusser1}). On the other hand, it was shown in \cite[Corollary 16]{preusser-1}, that the class of weighted Leavitt path algebras contains infinitely many domains which are neither isomorphic to an unweighted Leavitt path algebra nor to a Leavitt algebra $L_K(n,n+k)$. 

As examples consider the weighted graphs 
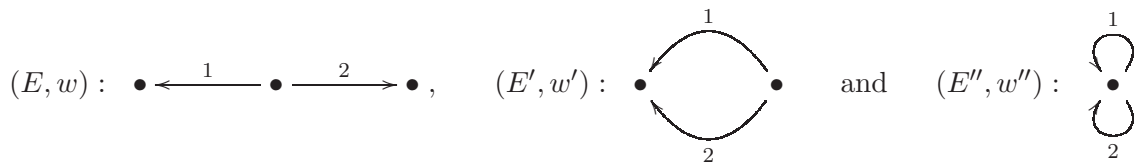
\[
(E,w):~\xymatrix@C+15pt{ \bullet& \bullet\ar[l]_{1}\ar[r]^{2}& \bullet},\hspace{0.7cm}(E',w'):~\xymatrix@C+15pt{ \bullet& \bullet\ar@/_1.7pc/[l]_{1}\ar@/^1.7pc/[l]^{2}} \hspace{0.5cm}\text{ and }\hspace{0.5cm} (E'',w''):\quad\xymatrix{ \bullet\ar@(ur,ul)_{1}\ar@(dr,dl)^{2}}~~
\]
where a number above or below an edge indicates the weight of that edge. In \cite[Example 40]{preusser} it was shown that $L_K(E,w)\cong L_K(F)$ where $F$ is the directed graph
\[
F: ~\vcenter{\vbox{
\xymatrix@C+15pt{ &\bullet\ar[d]&\\\bullet& \bullet\ar[l]\ar[r]& \bullet.}
}}
\]
In \cite[Example 21]{preusser-1} it was shown that $L_K(E',w')\cong L_K(F')$ where $F'$ is the directed graph
\vspace{0.7cm}
\[
F':~
\xymatrix@C+15pt{\bullet& \bullet\ar@/^1.7pc/[l]\ar@/_1.7pc/[l]\ar@/^1.7pc/[r]& \bullet.\ar@/^1.7pc/[l]}
\]
$~$\\
But it remained unclear if $L_K(E'',w'')$ is isomorphic to an unweighted Leavitt path algebra. It will follow from the results of this paper that $L_K(E'',w'')$ cannot be isomorphic to an unweighted Leavitt path algebra.

In this paper we obtain a graph-theoretic criterion that is sufficient and necessary for $L_K(E,w)$ being isomorphic to an unweighted Leavitt path algebra (Condition (LPA), cf. Definition \ref{defLPA}). Moreover, we prove that if $L_K(E,w)$ is Artinian, or von Neumann regular, or has finite Gelfand-Kirillov dimension, then $L_K(E,w)$ is isomorphic to an unweighted Leavitt path algebra.

The rest of the paper is organised as follows.

In Section 2 we recall some standard notation which is used throughout the paper.

In Section 3 we recall the definitions of the unweighted and weighted Leavitt path algebras. 

In Section 4 we introduce Condition (LPA).

In Section 5 we prove that if $(E,w)$ is a row-finite weighted graph that satisfies Condition (LPA), then $L_K(E,w)$ is isomorphic to an unweighted Leavitt path algebra. 

In Section 6 we prove that if $(E,w)$ is a row-finite weighted graph that does not satisfy Condition (LPA), then $L_K(E,w)$ is not isomorphic to an unweighted Leavitt path algebra. Moreover, we prove that if $L_K(E,w)$ is Artinian, or von Neumann regular, or has finite Gelfand-Kirillov dimension, then $L_K(E,w)$ is isomorphic to an unweighted Leavitt path algebra. We also prove  again that if $L_K(E,w)$ is locally finite or Noetherian, then $L_K(E,w)$ is isomorphic to an unweighted Leavitt path algebra (that has already been shown in \cite{preusser1}, but the paper was never published in a journal). 

In Section 7 we summarise the main results of this paper.

\section{Notation}
Throughout the paper $K$ denotes a field. By a $K$-algebra we mean an associative (but not necessarily commutative or unital) $K$-algebra. By an ideal we mean a two-sided ideal. $\N$ denotes the set of positive integers, $\N_0$ the set of nonnegative integers, $\Z$ the set of integers and $\R_+$ the set of positive real numbers. 

\section{Unweighted and weighted Leavitt path algebras}

\begin{definition}\label{defdg}
A {\it (directed) graph} is a quadruple $E=(E^0,E^1,s,r)$ where $E^0$ and $E^1$ are sets and $s,r:E^1\rightarrow E^0$ maps. The elements of $E^0$ are called {\it vertices} and the elements of $E^1$ {\it edges}. If $e$ is an edge, then $s(e)$ is called its {\it source} and $r(e)$ its {\it range}.
\end{definition}

\begin{remark}$~$\\
\vspace{-0.6cm}
\begin{enumerate}[(a)]
\item Let $E$ be a graph, $v\in E^0$ a vertex and $e\in E^1$ an edge. Then we say that $v$ {\it emits} $e$ if $s(e)=v$ and $v$ {\it receives} $e$ if $r(e)=v$. 
\item In this article all graphs are assumed to be row-finite. Recall that a graph $E=(E^0,E^1,s,r)$ is called {\it row-finite} if $s^{-1}(v)$ is a finite set for any vertex $v$.
\end{enumerate}
\end{remark}

\begin{definition}\label{deflpa}
Let $E$ be a graph. The $K$-algebra $L_K(E)$ presented by the generating set $\{v,e,e^*\mid v\in E^0,e\in E^1\}$ and the relations
\begin{enumerate}[(i)]
\item $uv=\delta_{uv}u\quad(u,v\in E^0)$,
\medskip
\item $s(e)e=e=er(e),~r(e)e^*=e^*=e^*s(e)\quad(e\in E^1)$,
\medskip
\item $e^*f= \delta_{ef}r(e)\quad(v\in E^0,e,f\in s^{-1}(v))$ and
\medskip
\item $\sum\limits_{e\in s^{-1}(v)}ee^*= v\quad(v\in E^0,s^{-1}(v)\neq\emptyset)$
\end{enumerate}
is called the {\it (unweighted) Leavitt path algebra} of $E$. 
\end{definition}

\begin{remark}\label{remunproplpa}
Let $E$ be a graph and $A$ a $K$-algebra that contains a set $X=\{\alpha_v, \beta_{e}, \gamma_{e}\mid v\in E^0, e\in E^1\}$ such that\\
\vspace{-0.1cm}
\begin{enumerate}[(i)]
\item the $\alpha_v$'s are pairwise orthogonal idempotents, 
\medskip
\item
$\alpha_{s(e)}\beta_{e}=\beta_{e}=\beta_{e}\alpha_{r(e)},~\alpha_{r(e)}\gamma_{e}=\gamma_{e}=\gamma_{e}\alpha_{s(e)}\quad(e\in E^1)$,
\medskip
\item $\gamma_{e}\beta_{f}= \delta_{ef}\alpha_{r(e)}\quad(v\in E^0,e,f\in s^{-1}(v))$ and
\medskip
\item $\sum\limits_{e\in s^{-1}(v)}\beta_{e}\gamma_{e}= \alpha_{v}\quad(v\in E^0,s^{-1}(v)\neq \emptyset)$.
\end{enumerate}
We call $X$ an {\it $E$-family} in $A$. By the relations defining $L_K(E)$, there exists a unique $K$-algebra homomorphism $\phi: L_K(E)\rightarrow A$ such that $\phi(v)=\alpha_v$, $\phi(e)=\beta_{e}$ and $\phi(e^*)=\gamma_{e}$ for all $v\in E^0$ and $e\in E^1$. We will refer to this as the {\it Universal Property of $L_K(E)$}.
\end{remark}

\begin{definition}
A {\it weighted graph} is a pair $(E,w)$ where $E$ is a graph and $w:E^1\rightarrow \N$ is a map. If $e\in E^1$, then $w(e)$ is called the {\it weight} of $e$. For a vertex $v\in E^0$ we set $w(v):=\max\{w(e)\mid e\in s^{-1}(v)\}$ with the convention $\max \emptyset=0$.
\end{definition}


\begin{definition}\label{def3}
Let $(E,w)$ be a weighted graph. The $K$-algebra $L_K(E,w)$ presented by the generating set $\{v,e_i,e_i^*\mid v\in E^0, e\in E^1, 1\leq i\leq w(e)\}$ and the relations
\begin{enumerate}[(i)]
\item $uv=\delta_{uv}u\quad(u,v\in E^0)$,
\medskip
\item $s(e)e_i=e_i=e_ir(e),~r(e)e_i^*=e_i^*=e_i^*s(e)\quad(e\in E^1, 1\leq i\leq w(e))$,
\medskip
\item 
$\sum\limits_{1\leq i\leq w(v)}e_i^*f_i= \delta_{ef}r(e)\quad(v\in E^0,e,f\in s^{-1}(v))$ and
\medskip 
\item $\sum\limits_{e\in s^{-1}(v)}e_ie_j^*= \delta_{ij}v\quad(v\in E^0,1\leq i, j\leq w(v))$
\end{enumerate}
is called the {\it weighted Leavitt path algebra} of $(E,w)$. In relations (iii) and (iv) we set $e_i$ and $e_i^*$ zero whenever $i > w(e)$. 
\end{definition}


\begin{example}\label{exex1}
If $(E,w)$ is a weighted graph such that $w(e)=1$ for all $e \in E^{1}$, then $L_K(E,w)$ is isomorphic to the unweighted Leavitt path algebra $L_K(E)$. 
\end{example}

\begin{example}\label{wlpapp}
Let $n\geq 1$ and $k\geq 0$. Let $(E,w)$ be the weighted graph
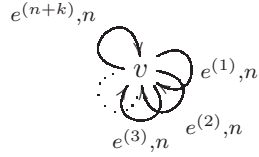
\begin{equation*}
\xymatrix{
& v \ar@{.}@(l,d) \ar@(ur,dr)^{e^{(1)},n} \ar@(r,d)^{e^{(2)},n} \ar@(dr,dl)^{e^{(3)},n} \ar@(l,u)^{e^{(n+k)},n}& 
}
\end{equation*}
with one vertex $v$ and $n+k$ edges $e^{(1)},\dots,e^{(n+k)}$ each of which has weight $n$. Then $L_K(E,w)$ is isomorphic to the Leavitt algebra $L_K(n,n+k)$, for details see \cite[Example 5.5]{hazrat13} or \cite[Example 4]{hazrat-preusser}. 
\end{example}

\begin{remark}\label{remunpropwlpa}
Let $(E,w)$ be a weighted graph and $A$ a $K$-algebra that contains a set $X=\{\alpha_v, \beta_{e,i}, \gamma_{e,i}\mid v\in E, e\in E^1, 1\leq i\leq w(e)\}$ such that\\
\vspace{-0.1cm}
\begin{enumerate}[(i)]
\item the $\alpha_v$'s are pairwise orthogonal idempotents, 
\medskip
\item
$\alpha_{s(e)}\beta_{e,i}=\beta_{e,i}=\beta_{e,i}\alpha_{r(e)},~\alpha_{r(e)}\gamma_{e,i}=\gamma_{e,i}=\gamma_{e,i}\alpha_{s(e)}\quad(e\in E^1, 1\leq i\leq w(e))$,
\medskip
\item $\sum\limits_{1\leq i\leq w(v)}\gamma_{e,i}\beta_{f,i}= \delta_{ef}\alpha_{r(e)}\quad(v\in E^0,e,f\in s^{-1}(v))$ and
\medskip
\item $\sum\limits_{e\in s^{-1}(v)}\beta_{e,i}\gamma_{e,j}= \delta_{ij}\alpha_{v}\quad(v\in E^0,1\leq i,j\leq w(v))$.
\end{enumerate}
In relations (iii) and (iv) we set $\beta_{e,i}$ and $\gamma_{e,i}$ zero whenever $i > w(e)$. We call $X$ an {\it $(E,w)$-family} in $A$. By the relations defining $L_K(E,w)$, there exists a unique $K$-algebra homomorphism $\phi: L_K(E,w)\rightarrow A$ such that $\phi(v)=\alpha_v$, $\phi(e_{i})=\beta_{e,i}$ and $\phi(e^*_{i})=\gamma_{e,i}$ for all $v\in E^0$, $e\in E^1$ and $1\leq i\leq w(e)$. We will refer to this as the {\it Universal Property of $L_K(E,w)$}.
\end{remark}

\begin{remark}
Let $(E,w)$ be a weighted graph. Then $L_K(E,w)$ has the properties (a)-(d) below, for details see \cite[Proposition 5.7]{hazrat13}.
\begin{enumerate}[(a)]
\item If $E^0$ is a finite set, then $L_K(E,w)$ is a unital ring (with $\sum\limits_{v\in E^0} v$ as multiplicative identity).
\item $L_K(E,w)$ has a set of local units, namely the set of all finite sums of distinct elements of $E^0$. Recall that an associative ring $R$ is said to have a {\it set of local units} $X$ in case $X$ is a set of idempotents in $R$ having the property that for each finite subset $S\subseteq R$ there exists an $x\in X$ such that $xsx=s$ for any $s\in S$.
\item There is an involution $*$ on $L_K(E,w)$ mapping $k\mapsto k$, $v\mapsto v$, $e_i\mapsto e_i^*$ and $e_i^*\mapsto e_i$ for any $k\in K$, $v\in E^0$, $e\in E^1$ and $1\leq i\leq w(e)$. 
\item Set $n:=\sup\{w(e) \mid e \in E^{1}\}$. One can define a $\mathbb Z^n$-grading on $L_K(E,w)$ by setting $\deg(v):=0$, $\deg(e_i):=\epsilon_i$ and $\deg(e_i^*):=-\epsilon_i$ for any $v\in E^0$, $e \in E^{1}$ and $1\leq i\leq w(e)$. Here $\epsilon_i$ denotes the element of $\mathbb Z^n$ whose $i$-th component is $1$ and whose other components are $0$.
\end{enumerate}
\end{remark}

\section{The Condition (LPA)}
We start with a couple of definitions.
\begin{definition}
Let $E$ be a graph. A {\it path} is a nonempty word $p=x_1\dots x_n$ over the alphabet $E^0\cup E^1$ such that either $x_i\in E^1~(i=1,\dots,n)$ and $r(x_i)=s(x_{i+1})~(i=1,\dots,n-1)$ or $n=1$ and $x_1\in E^0$. By definition, the {\it length} $|p|$ of $p$ is $n$ in the first case and $0$ in the latter case. 
We set $s(p):=s(x_1)$ and $r(p):=r(x_n)$ (here we use the convention $s(v)=v=r(v)$ for any $v\in E^0$).
\end{definition}

\begin{definition}
Let $E$ be a graph and $v\in E^0$. A {\it closed path (based at $v$)} is a path $p$ such that $|p|>0$ and $s(p)=r(p)=v$. A {\it cycle (based at $v$)} is a closed path $p=x_1\dots x_n$ based at $v$ such that $s(x_i)\neq s(x_j)$ for any $i\neq j$. 
\end{definition}

\begin{definition}\label{deftre}
Let $E$ be a graph. If $u,v\in E^0$ and there is a path $p$ in $E$ such that $s(p)=u$ and $r(p)=v$, then we write $u\geq v$. 
If $u\in E^0$, then $T(u):=\{v\in E^0 \mid u\geq v\}$ is called the {\it tree} of $u$. If $X\subseteq E^0$, we define $T(X):=\bigcup\limits_{v\in X}T(v)$. 
Two edges $e,f\in E^1$ are called {\it in line} if $e=f$ or $r(e)\geq s(f)$ or $r(f)\geq s(e)$
\end{definition}



\begin{definition}
Let $(E,w)$ be a weighted graph. An edge $e\in E^{1}$ is called {\it unweighted} if $w(e)=1$ and {\it weighted} if $w(e)>1$. The subset of $E^1$ consisting of all unweighted edges is denoted by $E_{uw}^1$ and the subset consisting of all weighted edges by $E_{w}^1$.
\end{definition}

Now we can introduce Condition (LPA).
\begin{definition}\label{defLPA}
We say that a weighted graph $(E,w)$ {\it satisfies Condition (LPA)} if the following holds true:
\begin{enumerate}[({LPA}1)]
\item Any vertex $v\in E^0$ emits at most one weighted edge.
\item Any vertex $v\in T(r(E^1_w))$ emits at most one edge.
\item If two weighted edges $e,f\in E^1_w$ are not in line, then $T(r(e))\cap T(r(f))=\emptyset$.
\item If $e\in E^1_w$ and $c$ is a cycle based at some vertex $v\in T(r(e))$, then $e$ belongs to $c$.
\end{enumerate}
\end{definition}

Each of the Conditions (LPA1)-(LPA4) in Definition \ref{defLPA} above ``forbids" a certain constellation in the weighted graph $(E,w)$. The pictures below illustrate these forbidden constellations. Symbols above or below edges indicate the weight. A dotted arrow stands for a path.
\begin{enumerate}[({LPA}1)]
\item \[\xymatrix@R-1.5pc{& \bullet\\\bullet \ar_{>1}[dr] \ar^{>1}[ur]& \\& \bullet.}\]
\item \[\xymatrix@R-1.5pc{& & & \bullet\\\bullet \ar^{>1}[r] & \bullet \ar@{.>}[r] & \bullet \ar[dr] \ar[ur]& \\& & & \bullet.}\]
\item \[\xymatrix{\bullet \ar^{>1}[r] & \bullet \ar@{.>}[r] & \bullet &\bullet \ar@{.>}[l] & \bullet. \ar_{>1}[l]}\]
\item \[\xymatrix@R-1pc@C-0.5pc{
		&&&\bullet \ar[r]&\bullet \ar[rd]&\\
		\bullet \ar^{>1}[r] & \bullet \ar@{.>}[r] &\bullet \ar[ru]&&&\bullet. \ar[ld]\\
		&&&\bullet \ar[lu]&\bullet \ar@{.}[l]}	
		\]
\end{enumerate}
\begin{remark}
Conditions (LPA1), (LPA2) and (LPA3) already appeared in \cite{preusser} and \cite{preusser1}. They were independently found by N. T. Phuc. Condition (LPA4) is new, this condition is slightly weaker then Condition (iv) in \cite[Definition 19]{preusser1}.
\end{remark}

\section{Presence of Condition (LPA)}
\begin{lemma}\label{lemLPA}
Let $(E,w)$ be a weighted graph that satisfies Condition (LPA). If $e$ and $f$ are distinct edges such that $s(e), s(f)\in T(r(E^1_w))$, then $r(e)\neq r(f)$.
\end{lemma}
\begin{proof}
Let $e,f\in E^1$ such that $s(e), s(f)\in T(r(E^1_w))$ and $r(e)=r(f)$. We will show that $e=f$. Since $s(e), s(f)\in T(r(E^1_w))$, there are $g,h\in E^1_w$ such that $s(e)\in T(r(g))$ and $s(f)\in T(r(h))$. It follows that $r(e)=r(f)\in T(r(g))\cap T(r(h))$. Since $(E,w)$ satisfies Condition (LPA3), $g$ and $h$ are in line. It follows that $s(e), s(f)\in T(r(g))$ or $s(e), s(f)\in T(r(h))$. W.l.o.g. assume that $s(e), s(f)\in T(r(g))$.
\begin{enumerate}
\item[Case 1] Assume that there is a cycle $c$ based at some vertex $v\in T(r(g))$. Since $(E,w)$ satisfies (LPA4), $g$ belongs to $c$. Write $c=\alpha^{(1)}\dots\alpha^{(n)}$ where $\alpha^{(1)},\dots,\alpha^{(n)}\in E^1$. Set $x_i:=s(\alpha^{(i)})~(1\leq i\leq n)$. Then, in view of (LPA2), we have $T(r(g))=\{x_1,\dots,x_n\}$. Moreover, each $x_i$ emits precisely one edge, namely $\alpha^{(i)}$. Since $s(e),s(f)\in T(r(g))$, we get that $s(e)=x_i$ and $s(f)=x_j$ for some $1\leq i,j\leq n$. Hence $e=\alpha^{(i)}$ and $f=\alpha^{(j)}$. Since $r(e)=r(f)$, it follows that $i=j$ and hence $e=f$.
\medskip
\item[Case 2] Assume that no cycle is based at a vertex in $T(r(g))$. Since $s(e),s(f)\in T(r(g))$, there are paths $p$ and $q$ such that $s(p)=r(g)=s(q)$, $r(p)=s(e)$ and $r(q)=s(f)$. Clearly $pe$ and $qf$ are paths starting at $r(g)$ and ending at $r(e)=r(f)$. It follows from (LPA2) and the assumption that no cycle is based at a vertex in $T(r(g))$, that $pe=qf$. Hence $e=f$. 
\end{enumerate}
\end{proof}

Recall that if $E$ is a graph, then a vertex $v$ that does not emit any edges is called a {\it sink}.
\begin{lemma}\label{lemkey1}
Let $(E,w)$ be a weighted graph that satisfies Condition (LPA). Then there is a weighted graph $(\tilde E,\tilde w)$ such that the ranges of the weighted edges in $(\tilde E,\tilde w)$ are sinks, no vertex in $(\tilde E,\tilde w)$ emits or receives two distinct weighted edges, and $L_K(\tilde E,\tilde w)\cong L_K(E,w)$.
\end{lemma}
\begin{proof}
Set $Z:=T(r(E^1_w))$. Define a weighted graph $(\tilde E,\tilde w)$ by $\tilde E^0=E^0$, $\tilde E^1=\tilde E^1_Z\sqcup \tilde E^1_{Z^c}$ where
\[\tilde E^1_Z=\{e^{(1)},\dots,e^{(w(e))}\mid e\in E^1, s(e)\in Z\}\text{ and }\tilde E^1_{Z^c}=\{e\mid e\in E^1,s(e)\not\in Z\},\]
$\tilde s(e^{(i)})=r(e)$, $\tilde r(e^{(i)})=s(e)$ and $\tilde w(e^{(i)})=1$ for any $e^{(i)}\in \tilde E^1_Z$ and $\tilde s(e)=s(e)$, $\tilde r(e)=r(e)$ and $\tilde w(e)=w(e)$ for any $e\in \tilde E^1_{Z^c}$. We have divided the rest of the proof into three parts. In Part I we show that the ranges of the weighted edges in $(\tilde E,\tilde w)$ are sinks, in Part II we show that no vertex in $(\tilde E,\tilde w)$ emits or receives two distinct weighted edges, and in Part III we show that $L_K(\tilde E,\tilde w)\cong L_K(E,w)$.\\
\\
{\bf Part I} Let $\tilde e\in \tilde E^1_w$. We will show that $\tilde r(\tilde e)$ is a sink in $(\tilde E,\tilde w)$. Clearly $\tilde e\in \tilde E^1_{Z^c}$ since all the edges in $\tilde E^1_{Z}$ have weight one in $(\tilde E,\tilde w)$. Hence there is an $e\in E^1, s(e)\not\in Z$ such that $\tilde e=e$. Clearly $w(e)=\tilde w(e)=\tilde w(\tilde e)>1$. 
Now suppose that there is an $\tilde f\in \tilde E^1$ such that $\tilde s(\tilde f)=\tilde r(\tilde e)$. 
\begin{enumerate}
\item[Case 1] Assume that $\tilde f\in \tilde E^1_Z$. Then there is an $f\in E^1, s(f)\in Z$ and an $i\in\{1,\dots,w(f)\}$ such that $\tilde f=f^{(i)}$ (note that $e\neq f$, since $s(e)\not\in Z$). It follows that $r(e)=\tilde r(e)=\tilde r(\tilde e)=\tilde s(\tilde f)=\tilde s(f^{(i)})=r(f)$. Since $s(f)\in Z=T(r(E^1_w))$, there is a $g\in E^1_w$ such that $s(f)\in T(r(g))$. It follows that $r(f)\in T(r(e))\cap T(r(g))$. Since $(E,w)$ satisfies Condition (LPA3), we get that $e$ and $g$ are in line and hence $e=g$ or $r(e)\geq s(g)$ or $r(g)\geq s(e)$.
\medskip
\begin{enumerate}
\item[Case 1.1] Assume that $e=g$. Since $s(f)\in T(r(g))=T(r(e))$, there is a path $p$ such that $s(p)=r(e)$ and $r(p)=s(f)$. Since $r(f)=r(e)$, we have a closed path $pf$ based at $r(e)$. That implies the existence of a cycle $c$ based at $r(e)$. Since $(E,w)$ satisfies (LPA4), $e$ belongs to $c$ and therefore $s(e)\in T(r(e))$. Now we get the contradiction $s(e)\in T(r(e))\subseteq T(r(E^1_w))=Z$. 
\medskip
\item[Case 1.2] Assume that $r(e)\geq s(g)$. Then there is a path $p$ such that $s(p)=r(e)$ and $r(p)=s(g)$. Since $s(f)\in T(r(g))$, there is a path $q$ such that $s(q)=r(g)$ and $r(q)=s(f)$. Since $r(f)=r(e)$, we have a closed path $pgqf$ based at $r(e)$. Now we can proceed as in Case 1.1 to get a contradiction.
\medskip
\item[Case 1.3] Assume that $r(g)\geq s(e)$. Then we get the contradiction $s(e)\in T(r(g))\subseteq T(r(E^1_w))=Z$. 
\end{enumerate}
\medskip
\item[Case 2] Assume that $\tilde f\in \tilde E^1_{Z^c}$. Then there is an $f\in E^1, s(f)\not\in Z$ such that $\tilde f=f$. It follows that $r(e)=\tilde r(e)=\tilde r(\tilde e)=\tilde s(\tilde f)=\tilde s(f)=s(f)$. Hence we get the contradiction $s(f)=r(e)\in T(r(E^1_w))\subseteq Z$.
\end{enumerate}
Thus the ranges of the weighted edges in $(\tilde E,\tilde w)$ are sinks. \\
\\
{\bf Part II} Assume that there are distinct $\tilde e,\tilde f\in \tilde E_w^1$ such that $\tilde s(\tilde e)=\tilde s(\tilde f)$. Clearly $\tilde e,\tilde f\in \tilde E^1_{Z^c}$ since all the edges in $\tilde E^1_{Z}$ have weight one in $(\tilde E,\tilde w)$. Hence there are distinct $e,f\in E^1, s(e),s(f)\not\in Z$ such that $\tilde e=e$ and $\tilde f=f$. It follows that $s(e)=\tilde s(e)=\tilde s(\tilde e)=\tilde s(\tilde f)=\tilde s(f)=s(f)$ which contradicts the assumption that $(E,w)$ satisfies Condition (LPA1) (note that $w(e)=\tilde w(\tilde e)>1$ and $w(f)=\tilde w(\tilde f)>1$). Thus no vertex emits two distinct weighted edges in $(\tilde E,\tilde w)$.\\
Now assume that there are distinct $\tilde e,\tilde f\in \tilde E_w^1$ such that $\tilde r(\tilde e)=\tilde r(\tilde f)$. Clearly $\tilde e,\tilde f\in \tilde E^1_{Z^c}$ since all the edges in $\tilde E^1_{Z}$ have weight one in $(\tilde E,\tilde w)$. Hence there are distinct $e,f\in E^1, s(e),s(f)\not\in Z$ such that $\tilde e=e$ and $\tilde f=f$. It follows that $r(e)=\tilde r(e)=\tilde r(\tilde e)=\tilde r(\tilde f)=\tilde r(f)=r(f)$. Since $(E,w)$ satisfies Condition (LPA3), we get that $e$ and $f$ are in line. Since $e$ and $f$ are distinct, it follows that $r(e)\geq s(f)$ or $r(f)\geq s(e)$. But in the first case we get the contradiction $s(f)\in Z$ and in the second case the contradiction $s(e)\in Z$. Thus no vertex receives two distinct weighted edges in $(\tilde E,\tilde w)$.\\
\\
{\bf Part III} It remains to show that $L_K(\tilde E,\tilde w)\cong L_K(E,w)$. Set $X:=\{v, e_i, e_i^*\mid v\in E^0, e\in E^1, 1\leq i\leq w(e)\}$ and $\tilde X:=\{\tilde v, \tilde e_i, \tilde e_i^*\mid \tilde v\in \tilde E^0, \tilde e\in \tilde E^1, 1\leq i\leq \tilde w(\tilde e)\}$. Let $K\X$ and $K\langle \tilde X\rangle$ be the free $K$-algebras generated by $X$ and $\tilde X$, respectively. Then the bijection $X\rightarrow \tilde X$ mapping 
\begin{alignat*}{2}
v&\mapsto v&\hspace{1cm}&(v\in E^0),\\
e_i&\mapsto (e^{(i)}_{1})^*&&(e\in E^1, s(e)\in Z,1\leq i\leq w(e)),\\
e_i^*&\mapsto e^{(i)}_{1}&&(e\in E^1, s(e)\in Z,1\leq i\leq w(e)),\\
e_i&\mapsto e_i&&(e\in E^1, s(e)\not\in Z,1\leq i\leq w(e)),\\
e_i^*&\mapsto e^*_{i}&&(e\in E^1, s(e)\not\in Z,1\leq i\leq w(e))
\end{alignat*}
induces an isomorphism $\phi:K\X\rightarrow K\langle \tilde X\rangle$. Let $I$ and $\tilde I$ be the ideals of $K\X$ and $K\langle \tilde X\rangle$ generated by the relations (i)-(iv) in Definition~\ref{def3}, respectively 
(hence $L_K(E,w)\cong K\X/I$ and $L_K(\tilde E, \tilde w)\cong K\langle \tilde X\rangle/\tilde I$). In order to show that $L_K(E,w)\cong L_K(\tilde E,\tilde w)$ it suffices to show that $\phi(I)=\tilde I$. Set
\[A^{(i)}:=\big \{uv-\delta_{uv}u \ |  \ u,v\in E^0 \big \},\]
\[A^{(ii)}:=\big \{s(e)e_i-e_i,~ e_ir(e)-e_i, ~r(e)e_i^*-e_i^*,~ e_i^*s(e)-e^*_i \ | \ e\in E^{1}, 1\leq i\leq w(e)\big \},\]
and for any $v\in E^0$ 
\[A^{(iii)}_v:=\Big\{\sum\limits_{1\leq i \leq w(v)}e_i^{*}f_i-\delta_{ef}r(e)\ | \ e,f\in s^{-1}(v)\Big\}\]
and
\[A^{(iv)}_v:=\Big\{\sum\limits_{e\in s^{-1}(v)}e_ie_j^*-\delta_{ij}v \ |  \ 1\leq i,j\leq w(v)\Big\}.\] 
Then $I$ is generated by $A^{(i)}$, $A^{(ii)}$, the $A^{(iii)}_v$'s and the $A^{(iv)}_v$'s. Analogously define subsets $B^{(i)},B^{(ii)},B^{(iii)}_v~(v\in \tilde E^0),B^{(iv)}_v~(v\in \tilde E^0)$ of $K\langle \tilde X\rangle$. Then $\tilde I$ is generated by $B^{(i)}$, $B^{(ii)}$, the $B^{(iii)}_v$'s and the $B^{(iv)}_v$'s. Clearly $\phi(A^{(i)})=B^{(i)}$ and $\phi(A^{(ii)})=B^{(ii)}$. One checks easily that $\phi(A_v^{(iii)})=B_v^{(iii)}$ and $\phi(A_v^{(iv)})=B_v^{(iv)}$ if $v\not\in Z$. \\
Let now $v\in Z$ be not a sink in $(E,w)$ (if $v\in Z$ is a sink in $(E,w)$, then $A_v^{(iii)}=A_v^{(iv)}=\emptyset$). Then we have $s^{-1}(v)=\{e\}$ for some $e\in E^1$ since $(E,w)$ satisfies Condition (LPA2). Set $\bar v:=r(e)$. Clearly
\[A^{(iii)}_v=\Big\{\sum\limits_{1\leq i\leq w(e)}e_i^{*}e_i-\bar v\Big\}\]
and
\[A^{(iv)}_v=\big \{e_ie_j^*-\delta_{ij}v  \  | \  1\leq i,j\leq w(e) \big \}.\]
It follows from Lemma \ref{lemLPA} that $\tilde s^{-1}(\bar v)=\{e^{(1)},\dots,e^{(w(e))}\}$. Hence 
\[B^{(iii)}_{\bar v}=\big  \{(e^{(i)}_1)^{*}e_1^{(j)}-\delta_{ij}v\  | \  1\leq i,j\leq w(e) \big  \}\]
and
\[B^{(iv)}_{\bar v}=\Big\{\sum\limits_{1\leq i\leq w(e)}e^{(i)}_1(e^{(i)}_1)^*-\bar v\Big\}\]
Clearly $\phi(A^{(iii)}_v)=B^{(iv)}_{\bar v}$ and $\phi(A^{(iv)}_v)=B^{(iii)}_{\bar v}$. It follows from Lemma \ref{lemLPA} that the map $~\bar{}~: v\mapsto \bar v$ defines a bijection between the elements of $Z$ that are not a sink in $(E,w)$ and the elements of $Z$ that are not a sink in $(\tilde E,\tilde w)$. Hence $\phi(I)=\tilde I$ and thus $L_K(\tilde E,\tilde w)\cong L_K(E,w)$.
\end{proof}

\begin{example}\label{exlpa1}
Consider the weighted graph
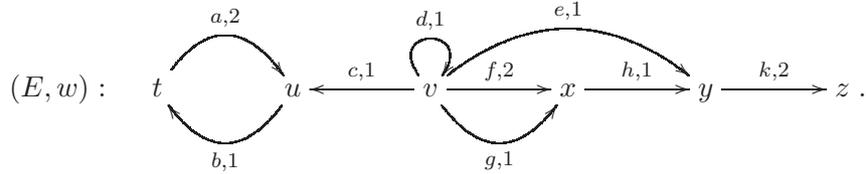
\[
(E,w):\quad\xymatrix@C+15pt{
t\ar@/^1.7pc/[r]^{a,2}&u\ar@/^1.7pc/[l]^{b,1}& v\ar[l]_{c,1}\ar@(ul,ur)^{d,1}\ar@/^1.9pc/[rr]^{e,1}\ar[r]^{f,2}\ar@/_1.7pc/[r]_{g,1}&x\ar[r]^{h,1}&y\ar[r]^{k,2}&z}.
\]
One checks easily that $(E,w)$ satisfies Condition (LPA) (note that $T(r(E^1_w))=\{t,u,x,y,z\}$). Let $(\tilde E,\tilde w)$ be defined as in the proof of Lemma \ref{lemkey1}. Then $(\tilde E,\tilde w)$ is the weighted graph
 \[
(\tilde E,\tilde w):\quad
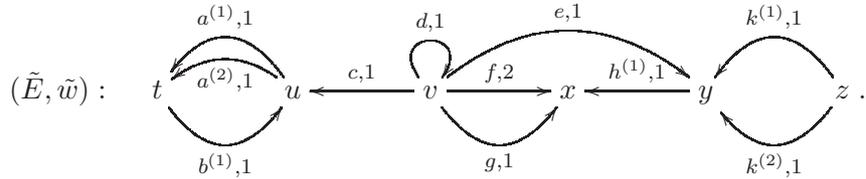
\xymatrix@C+15pt{
t\ar@/_1.7pc/[r]_{b^{(1)},1}&u\ar@/_1.7pc/[l]_{a^{(1)},1}\ar@/_1.0pc/[l]^{a^{(2)},1}& v\ar[l]_{c,1}\ar@(ul,ur)^{d,1}\ar@/^1.9pc/[rr]^{e,1}\ar[r]^{f,2}\ar@/_1.7pc/[r]_{g,1}&x&y\ar[l]_{h^{(1)},1}&z\ar@/_1.7pc/[l]_{k^{(1)},1}\ar@/^1.7pc/[l]^{k^{(2)},1}
}.
\]
There is only one weighted edge in $(\tilde E,\tilde w)$, namely $f$, and its range is a sink. The proof of Lemma \ref{lemkey1} shows that $L_K(E,w)\cong L_K(\tilde E,\tilde w)$.
\end{example}

\begin{lemma}\label{lemkey2}
Let $(E,w)$ be a weighted graph such that the ranges of the weighted edges are sinks and no vertex emits or receives two distinct weighted edges. Then there is a graph $\tilde E$ such that $L_K(E,w)\cong L_K(\tilde E)$.
\end{lemma}
\begin{proof}
If $v\in r(E^1_w)$, then there is a unique edge $g^v\in E^1_w$ such that $r(g^v)=v$ (since no vertex in $(E,w)$ receives two distinct weighted edges). Define a graph $\tilde E$ by 
\begin{align*}
\tilde E^0&=M\sqcup N \text{ where }\\
M&=E^0\setminus r(E^1_w),\\
N&=\{v^{(1)},\dots,v^{(w(g^v))}\mid v\in r(E^1_w)\},\\
\tilde E^1&=A\sqcup B\sqcup C \sqcup D\text{ where }\\
A&=\{e\mid e\in E^1_{uw},r(e)\not\in r(E^1_w)\},\\ 
B&=\{e^{(1)},\dots,e^{(w(g^{r(e)}))}\mid e\in E^1_{uw}, r(e)\in r(E^1_w)\},\\
C&=\{e^{(1)}\mid e\in E^1_{w}\},\\
D&=\{e^{(2)},\dots,e^{(w(e))}\mid e\in E^1_{w}\},\\
\tilde s(e)&=s(e),~\tilde r(e)=r(e)\quad (e\in A),\\
\tilde s(e^{(i)})&=s(e),~\tilde r(e^{(i)})=r(e)^{(i)}\quad (e^{(i)}\in B),\\
\tilde s(e^{(1)})&=s(e),~\tilde r(e^{(1)})=r(e)^{(1)}\quad (e^{(1)}\in C),\\
\tilde s(e^{(i)})&=r(e)^{(i)},~\tilde r(e^{(i)})=s(e)\quad (e^{(i)}\in D),
\end{align*}
(note that if $e\in E^1$, then $s(e)\in E^0\setminus r(E^1_w)$ since the elements of $r(E^1_w)$ are sinks).
We have divided the rest of the proof into three parts. In Part I we define a homomorphism $\phi:L_K(E,w)\rightarrow L_K(\tilde E)$, in Part II we define a homomorphism $\tilde\phi:L_K(\tilde E)\rightarrow L_K(E,w)$, and in Part III we show that $\phi$ and $\tilde\phi$ are inverse to each other.\\
\\
{\bf Part I} Set
\begin{align*}
\alpha_v&:=\begin{cases}v,&\text{ if } v\not\in r(E^1_w),\\\sum\limits_{i=1}^{w(g^v)}v^{(i)},&\text{ if }v\in r(E^1_w),\end{cases}\\
\beta_{e,i}&:=\begin{cases}e,&\text{ if } e\in E_{uw}^1,r(e)\not\in r(E^1_w),i=1,\\\sum\limits_{j=1}^{w(g^{r(e)})}e^{(j)},&\text{ if }e\in E_{uw}^1,r(e)\in r(E^1_w),i=1,\\
e^{(1)},&\text{ if }e\in E_{w}^1,i=1,\\
(e^{(i)})^*,&\text{ if }e\in E_{w}^1,i>1,
\end{cases}\\
\gamma_{e,i}&:=\begin{cases}e^*,&\text{ if } e\in E_{uw}^1,r(e)\not\in r(E^1_w),i=1,\\\sum\limits_{j=1}^{w(g^{r(e)})}(e^{(j)})^*,&\text{ if }e\in E_{uw}^1,r(e)\in r(E^1_w),i=1,\\
(e^{(1)})^*,&\text{ if }e\in E_{w}^1,i=1,\\
e^{(i)},&\text{ if }e\in E_{w}^1,i>1.
\end{cases}
\end{align*}
In order to show that $X:=\{\alpha_v,\beta_{e,i},\gamma_{e_i}\mid v\in E^0,e\in E^1, 1\leq i\leq w(e)\}$ is an $(E,w)$-family in $L_K(\tilde E)$, one has to show that the relations (i)-(iv) in Remark \ref{remunpropwlpa} are satisfied. We leave (i) and (ii) to the reader and show (iii) and (iv).\\ 
First we check (iii). Let $v\in E^0$ and $e,f\in s^{-1}(v)$. We have to show that $\sum\limits_{1\leq i\leq w(v)}\gamma_{e,i}\beta_{f,i}=\delta_{ef}\alpha_{r(e)}$. 
\begin{enumerate}
\item[Case 1] Assume that $e,f\in E^1_{uw}$. 
\medskip
\begin{enumerate}
\item[Case 1.1] Assume that $r(e),r(f)\not\in r(E^1_w)$. Then 
\[\sum\limits_{1\leq i\leq w(v)}\gamma_{e,i}\beta_{f,i}=e^*f=\delta_{ef}\tilde r(e)=\delta_{ef}r(e)=\delta_{ef}\alpha_{r(e)}.\]
\medskip
\item[Case 1.2] Assume that $r(e)\not\in r(E^1_w)$ and $r(f)\in r(E^1_w)$. Then 
\[\sum\limits_{1\leq i\leq w(v)}\gamma_{e,i}\beta_{f,i}=e^*\sum\limits_{j=1}^{w(g^{r(f)})}f^{(j)}=\sum\limits_{j=1}^{w(g^{r(f)})}e^*f^{(j)}=0=\delta_{ef}\alpha_{r(e)}.\]
\medskip
\item[Case 1.3] Assume that $r(e)\in r(E^1_w)$ and $r(f)\not\in r(E^1_w)$. Then 
\[\sum\limits_{1\leq i\leq w(v)}\gamma_{e,i}\beta_{f,i}=\sum\limits_{j=1}^{w(g^{r(e)})}(e^{(j)})^*f=0=\delta_{ef}\alpha_{r(e)}.\]
\medskip
\item[Case 1.4] Assume that $r(e),r(f)\in r(E^1_w)$. Then 
\[\sum\limits_{1\leq i\leq w(v)}\gamma_{e,i}\beta_{f,i}=\sum\limits_{j=1}^{w(g^{r(e)})}(e^{(j)})^*\sum\limits_{k=1}^{w(g^{r(f)})}f^{(k)}=\sum\limits_{j=1}^{w(g^{r(e)})}\sum\limits_{k=1}^{w(g^{r(f)})}(e^{(j)})^*f^{(k)}=\delta_{ef}\sum\limits_{j=1}^{w(g^{r(e)})}r(e)^{(j)}=\delta_{ef}\alpha_{r(e)}.\]
\end{enumerate}
\medskip
\item[Case 2] Assume that $e\in E^1_{uw}$ and $f\in E^1_w$.
\medskip
\begin{enumerate}
\item[Case 2.1] Assume that $r(e)\not\in r(E^1_w)$. Then 
\[\sum\limits_{1\leq i\leq w(v)}\gamma_{e,i}\beta_{f,i}=e^*f^{(1)}=0=\delta_{ef}\alpha_{r(e)}.\]
\medskip
\item[Case 2.2] Assume that $r(e)\in r(E^1_w)$. Then 
\[\sum\limits_{1\leq i\leq w(v)}\gamma_{e,i}\beta_{f,i}=\sum\limits_{j=1}^{w(g^{r(e)})}(e^{(j)})^*f^{(1)}=0=\delta_{ef}\alpha_{r(e)}.\]
\end{enumerate}
\medskip
\item[Case 3] Assume that $e\in E^1_{w}$ and $f\in E^1_{uw}$. This case is similar to Case 2 and therefore is ommitted.
\medskip
\item[Case 4] Assume that $e,f\in E^1_w$. Since no vertex emits two distinct weighted edges in $(E,w)$, it follows that $e=f$ and $w(v)=w(e)$. Clearly
\[\sum\limits_{1\leq i\leq w(v)}\gamma_{e,i}\beta_{f,i}=(e^{(1)})^*e^{(1)}+\sum\limits_{j=2}^{w(e)}e^{(j)}(e^{(j)})^*=r(e)^{(1)}+\sum\limits_{j=2}^{w(e)} r(e)^{(j)}=\delta_{ef}\alpha_{r(e)}\]
(note that $e^{(j)}$ is the only edge emitted by $r(e)^{(j)}$ in $\tilde E$).
\end{enumerate}\medskip
Thus (iii) holds. \\
Next we check (iv). Let $v\in E^0$ and $1\leq i,j\leq w(v)$. Note that the existence of $i,j$ with the property $1\leq i,j\leq w(v)$ implies that $w(v)\geq 1$, i.e. that $v$ is not a sink in $(E,w)$. It follows that $v\in E^0\setminus r(E^1_w)$. We have to show that $\sum\limits_{e\in s^{-1}(v)}\beta_{e,i}\gamma_{e,j}= \delta_{ij}\alpha_{v}$.

\begin{enumerate}[C{a}se (a)]
\item Assume that $i=j=1$. Clearly
\begin{align*}
&\sum\limits_{e\in s^{-1}(v)}\beta_{e,1}\gamma_{e,1}\\
=&\sum\limits_{\substack{e\in s^{-1}(v)\cap E^1_{uw},\\r(e)\not\in r(E^1_w)}}\beta_{e,1}\gamma_{e,1}+\sum\limits_{\substack{e\in s^{-1}(v)\cap E^1_{uw},\\r(e)\in r(E^1_w)}}\beta_{e,1}\gamma_{e,1}+\sum\limits_{e\in s^{-1}(v)\cap E^1_{w}}\beta_{e,1}\gamma_{e,1}\\
=&\sum\limits_{\substack{e\in s^{-1}(v)\cap E^1_{uw},\\r(e)\not\in r(E^1_w)}}ee^*+\sum\limits_{\substack{e\in s^{-1}(v)\cap E^1_{uw},\\r(e)\in r(E^1_w)}}\sum\limits_{j=1}^{w(g^{r(e)})}e^{(j)}\sum\limits_{k=1}^{w(g^{r(e)})}(e^{(k)})^*+\sum\limits_{e\in s^{-1}(v)\cap E^1_{w}}e^{(1)}(e^{(1)})^*\\
=&\underbrace{\sum\limits_{\substack{e\in s^{-1}(v)\cap E^1_{uw},\\r(e)\not\in r(E^1_w)}}ee^*+\sum\limits_{\substack{e\in s^{-1}(v)\cap E^1_{uw},\\r(e)\in r(E^1_w)}}\sum\limits_{j,k=1}^{w(g^{r(e)})}e^{(j)}(e^{(k)})^*+\sum\limits_{e\in s^{-1}(v)\cap E^1_{w}}e^{(1)}(e^{(1)})^*}_{T_1:=}.
\end{align*}
Since $\tilde r(e^{(j)})=r(e)^{(j)}$ for any $e\in E^1_{uw},r(e)\in r(E^1_w)$, we have $e^{(j)}(e^{(k)})^*=0$ in $L_K(\tilde E)$ whenever $j\neq k$. Hence
\[T_1=\underbrace{\sum\limits_{\substack{e\in s^{-1}(v)\cap E^1_{uw},\\r(e)\not\in r(E^1_w)}}ee^*+\sum\limits_{\substack{e\in s^{-1}(v)\cap E^1_{uw},\\r(e)\in r(E^1_w)}}\sum\limits_{j=1}^{w(g^{r(e)})}e^{(j)}(e^{(j)})^*+\sum\limits_{e\in s^{-1}(v)\cap E^1_{w}}e^{(1)}(e^{(1)})^*}_{T_2:=}.\]
One checks easily that 
\begin{align*}
\tilde s^{-1}(v)=&\{e\mid e\in s^{-1}(v)\cap E^1_{uw},r(e)\not\in r(E^1_w)\}\\
&\sqcup\{e^{(j)}\mid e\in s^{-1}(v)\cap E^1_{uw}, r(e)\in r(E^1_w),1\leq j\leq w(g^{r(e)})\}\\
&\sqcup\{e^{(1)}\mid e\in s^{-1}(v)\cap E^1_{w}\}.
\end{align*}
Hence $T_2=v=\delta_{11}\alpha_{v}$.
\medskip
\item Assume that $i=1$ and $j>1$. Then $w(v)\geq j>1$ and hence $v$ emits precisely one weighted edge $f$. Since $\gamma_{e,j}=0$ whenever $j\geq w(e)$, we have
\[\sum\limits_{e\in s^{-1}(v)}\beta_{e,1}\gamma_{e,j}=\beta_{f,1}\gamma_{f,j}=f^{(1)}f^{(j)}=0=\delta_{1j}\alpha_{v}\]
(note that $\tilde r(f^{(1)})=r(f)^{(1)}\neq r(f)^{(j)}=\tilde s(f^{(j)})$).
\medskip
\item Assume that $i>1$ and $j=1$. Then $v$ emits precisely one weighted edge $f$. Clearly
\[\sum\limits_{e\in s^{-1}(v)}\beta_{e,i}\gamma_{e,1}=\beta_{f,i}\gamma_{f,1}=(f^{(i)})^*(f^{(1)})^*=0=\delta_{i1}\alpha_{v}\]
(note that $\tilde s(f^{(i)})=r(f)^{(i)}\neq r(f)^{(1)}=\tilde r(f^{(1)})$).
\medskip
\item Assume that $i,j>1$. Then $v$ emits precisely one weighted edge $f$. Clearly
\[\sum\limits_{e\in s^{-1}(v)}\beta_{e,i}\gamma_{e,j}=\beta_{f,i}\gamma_{f,j}=(f^{(i)})^*f^{(j)}=\delta_{ij}\tilde r(f^{(i)})=\delta_{ij}v=\delta_{ij}\alpha_{v}.\]
\end{enumerate}\medskip
Thus (iv) holds too and hence $X$ is an $(E,w)$-family in $L_K(\tilde E)$. By the Universal Property of $L_K(E,w)$ there is a unique $K$-algebra homomorphism $\phi: L_K(E,w)\rightarrow L_K(\tilde E)$ such that $\phi(v)=\alpha_v$, $\phi(e_{i})=\beta_{e,i}$ and $\phi(e^*_{i})=\gamma_{e,i}$ for all $v\in E^0$, $e\in E^1$ and $1\leq i\leq w(e)$.\\
\\
{\bf Part II}
Set
\begin{align*}
\tilde \alpha_{\tilde v}&:=\begin{cases}v,&\text{ if } \tilde v=v\in M,\\(g^v_{i})^*g^v_i,&\text{ if }\tilde v=v^{(i)}\in N,\end{cases}\\
\tilde\beta_{\tilde e}&:=\begin{cases}e_1,&\text{ if } \tilde e=e\in A,\\e_1(g^{r(e)}_{i})^*g^{r(e)}_i,&\text{ if }\tilde e=e^{(i)}\in B,\\
e_1,&\text{ if }\tilde e=e^{(1)}\in C,\\
e_i^*,&\text{ if }\tilde e=e^{(i)}\in D,
\end{cases}\\
\tilde\gamma_{\tilde e}&:=\begin{cases}e_1^*,&\text{ if } \tilde e=e\in A,\\
(g^{r(e)}_{i})^*g^{r(e)}_ie_1^*,&\text{ if }\tilde e=e^{(i)}\in B,\\
e_1^*,&\text{ if }\tilde e=e^{(1)}\in C,\\
e_i,&\text{ if }\tilde e=e^{(i)}\in D.
\end{cases}
\end{align*}
In order to show that $\tilde X:=\{\tilde\alpha_{\tilde v},\tilde\beta_{\tilde e},\tilde\gamma_{\tilde e}\mid \tilde v\in \tilde E^0,\tilde e\in \tilde E^1\}$ is an $\tilde E$-family in $L_K(E,w)$, one has to show that the relations (i)-(iv) in Remark \ref{remunproplpa} are satisfied. We leave (i) and (ii) to the reader and show (iii) and (iv).\\ 
First we check (iii). Let $\tilde v\in \tilde E^0$ and $\tilde e, \tilde f\in \tilde s^{-1}(\tilde v)$. We have to show that $\tilde\gamma_{\tilde e}\tilde\beta_{\tilde f}=\delta_{\tilde e \tilde f}\tilde\alpha_{\tilde r(\tilde e)}$. 

\begin{enumerate}
\item[Case 1] Assume that $\tilde v\in M$. Then $\tilde e,\tilde f\in A\cup B\cup C$ since $\tilde s^{-1}(D)\subseteq N$.
\medskip
\begin{enumerate}
\item[Case 1.1] Assume that $\tilde e,\tilde f\in A$. Then there are $e,f\in E^1_{uw}, r(e),r(f)\not\in r(E^1_w)$ such that $\tilde e=e$ and $\tilde f=f$. Clearly
$\tilde\gamma_{\tilde e}\tilde\beta_{\tilde f}=e_1^*f_1=\delta_{ef}r(e)=\delta_{\tilde e \tilde f}\tilde\alpha_{\tilde r(\tilde e)}$.\medskip
\item[Case 1.2] Assume that $\tilde e\in A$ and $\tilde f\in B$. Then there is an $e\in E^1_{uw}, r(e)\not\in r(E^1_w)$ such that $\tilde e=e$. Moreover, there is an $f\in E^1_{uw}, r(f)\in r(E^1_w)$ and an $1\leq i \leq w(g^{r(f)})$ such that $\tilde f=f^{(i)}$. Clearly $e\neq f$ and $\tilde e\neq\tilde f$. Hence
$\tilde\gamma_{\tilde e}\tilde\beta_{\tilde f}=e_1^*f_1(g_i^{r(f)})^*g_i^{r(f)}=\delta_{ef}(g_i^{r(f)})^*g_i^{r(f)}=0=\delta_{\tilde e \tilde f}\tilde\alpha_{\tilde r(\tilde e)}$.\medskip
\item[Case 1.3] Assume that $\tilde e\in A$ and $\tilde f\in C$. Then there is an $e\in E^1_{uw}, r(e)\not\in r(E^1_w)$ such that $\tilde e=e$. Moreover, there is an $f\in E^1_{w}$ such that $\tilde f=f^{(1)}$. Clearly $e\neq f$ and $\tilde e\neq\tilde f$. Hence $\tilde\gamma_{\tilde e}\tilde\beta_{\tilde f}=e_1^*f_1=\delta_{ef}r(e)=0=\delta_{\tilde e \tilde f}\tilde\alpha_{\tilde r(\tilde e)}$.\medskip
\item[Case 1.4] Assume that $\tilde e\in B$ and $\tilde f\in A$. Then there is an $e\in E^1_{uw}, r(e)\in r(E^1_w)$ and an $1\leq i \leq w(g^{r(e)})$ such that $\tilde e=e^{(i)}$. Moreover, there there is an $f\in E^1_{uw}, r(f)\not\in r(E^1_w)$ such that $\tilde f=f$. Clearly $e\neq f$ and $\tilde e\neq \tilde f$. Hence
$\tilde\gamma_{\tilde e}\tilde\beta_{\tilde f}=(g^{r(e)}_{i})^*g^{r(e)}_ie_1^*f_1=\delta_{ef}(g^{r(e)}_{i})^*g^{r(e)}_i=0=\delta_{\tilde e \tilde f}\tilde\alpha_{\tilde r(\tilde e)}$.\medskip
\item[Case 1.5] Assume that $\tilde e,\tilde f\in B$. Then there are $e,f\in E^1_{uw}, r(e),r(f)\in r(E^1_w)$ and $1\leq i\leq w(g^{r(e)}), 1\leq j\leq w(g^{r(f)})$ such that $\tilde e=e^{(i)}$ and $\tilde f=f^{(j)}$. Clearly $\tilde\gamma_{\tilde e}\tilde\beta_{\tilde f}=(g^{r(e)}_{i})^*g^{r(e)}_ie_1^*f_1(g_j^{r(f)})^*g_j^{r(f)}=\delta_{ef}(g^{r(e)}_{i})^*g^{r(e)}_i(g_j^{r(f)})^*g_j^{r(f)}=\delta_{ef}\delta_{ij}(g^{r(e)}_{i})^*g^{r(e)}_i=\delta_{\tilde e \tilde f}\tilde\alpha_{\tilde r(\tilde e)}$.
\medskip
\item[Case 1.6] Assume that $\tilde e\in B$ and $\tilde f\in C$. Then there is an $e\in E^1_{uw}, r(e)\in r(E^1_w)$ and a $1\leq i\leq w(g^{r(e)})$ such that $\tilde e=e^{(i)}$. Moreover, there is an $f\in E^1_{w}$ such that $\tilde f=f^{(1)}$. Clearly $e\neq f$ and $\tilde e\neq\tilde f$. Hence $\tilde\gamma_{\tilde e}\tilde\beta_{\tilde f}=(g^{r(e)}_{i})^*g^{r(e)}_ie_1^*f_1=\delta_{ef}(g^{r(e)}_{i})^*g^{r(e)}_i=0=\delta_{\tilde e \tilde f}\tilde\alpha_{\tilde r(\tilde e)}$.\medskip
\item[Case 1.7] Assume that $\tilde e\in C$ and $\tilde f\in A$. Then there is an $e\in E^1_w$ such that $\tilde e=e^{(1)}$. Moreover, there is an $f\in E^1_{uw}, r(f)\not\in r(E^1_w)$ such that $\tilde f=f$. Clearly $e\neq f$ and $\tilde e\neq\tilde f$. Hence
$\tilde\gamma_{\tilde e}\tilde\beta_{\tilde f}=e_1^*f_1=\delta_{ef}r(e)=0=\delta_{\tilde e \tilde f}\tilde\alpha_{\tilde r(\tilde e)}$.\medskip
\item[Case 1.8] Assume that $\tilde e\in C$ and $\tilde f\in B$. Then there is an $e\in E^1_w$ such that $\tilde e =e^{(1)}$. Moreover, there is an $f\in E^1_{uw}, r(f)\in r(E^1_w)$ and an $1\leq i \leq w(g^{r(f)})$ such that $\tilde f=f^{(i)}$. Clearly $e\neq f$ and $\tilde e\neq\tilde f$. Hence
$\tilde\gamma_{\tilde e}\tilde\beta_{\tilde f}=e_1^*f_1(g_i^{r(f)})^*g_i^{r(f)}=\delta_{ef}(g_i^{r(f)})^*g_i^{r(f)}=0=\delta_{\tilde e \tilde f}\tilde\alpha_{\tilde r(\tilde e)}$.\medskip
\item[Case 1.9] Assume that $\tilde e,\tilde f\in C$. Then there are $e,f\in E^1_{w}$ such that $\tilde e=e^{(1)}$ and $\tilde f=f^{(1)}$. Since $s(e)=\tilde s(e^{(1)})=\tilde s(\tilde e)=\tilde s(\tilde f)=\tilde s(f^{(1)})=s(f)$, we have $e=f$ (because no vertex in $(E,w)$ emits two distinct weighted edges). It follows that $\tilde e=\tilde f$. Clearly $\tilde\gamma_{\tilde e}\tilde\beta_{\tilde f}=e_1^*e_1=\delta_{\tilde e \tilde f}\tilde\alpha_{\tilde r(\tilde e)}$.
\end{enumerate}
\medskip
\item[Case 2] Assume that $\tilde v\in N$. Then $\tilde v=v^{(i)}$ for some $v\in r(E^1_w)$ and $1\leq i\leq w(g^v)$. One checks easily that $\tilde s^{-1}(\tilde v)=\emptyset$ if $i=1$ and $\tilde s^{-1}(\tilde v)=\{(g^v)^{(i)}\}$ if $i>1$. It follows that $i>1$ and $\tilde e=\tilde f=(g^v)^{(i)}$. Hence $\tilde\gamma_{\tilde e}\tilde\beta_{\tilde f}=g^v_i(g^v_i)^*=s(g^v)=\delta_{\tilde e \tilde f}\tilde\alpha_{\tilde r(\tilde e)}$.
\end{enumerate}
Thus (iii) holds. \\
Next we check (iv). Let $\tilde v\in \tilde E^0$ such that $\tilde s^{-1}(v)\neq\emptyset$. We have to show that $\sum\limits_{\tilde e\in \tilde s^{-1}(\tilde v)}\tilde\beta_{\tilde e}\tilde\gamma_{\tilde e}= \tilde\alpha_{\tilde v}$.
\begin{enumerate}[C{a}se (a)]
\item Assume that $\tilde v\in M$. Then $\tilde v=v$ for some $v\in E^0\setminus r(E^1_w)$. One checks easily that 
\begin{align*}
\tilde s^{-1}(\tilde v)=&\{e\mid e\in s^{-1}(v)\cap E^1_{uw},r(e)\not\in r(E^1_w)\}\\
&\sqcup\{e^{(i)}\mid e\in s^{-1}(v)\cap E^1_{uw}, r(e)\in r(E^1_w),1\leq i\leq w(g^{r(e)})\}\\
&\sqcup\{e^{(1)}\mid e\in s^{-1}(v)\cap E^1_{w}\}.
\end{align*}
Hence
\begin{align*}
&\sum\limits_{\tilde e\in \tilde s^{-1}(\tilde v)}\tilde \beta_{\tilde e}\tilde \gamma_{\tilde e}\\
=&\sum\limits_{\substack{e\in s^{-1}(v)\cap E^1_{uw},\\r(e)\not\in r(E^1_w)}}\tilde\beta_{e}\tilde\gamma_{e}+\sum\limits_{\substack{e\in s^{-1}(v)\cap E^1_{uw},\\r(e)\in r(E^1_w)}}\sum\limits_{i=1}^{w(g^{r(e)})}\tilde\beta_{e^{(i)}}\tilde\gamma_{e^{(i)}}+\sum\limits_{e\in s^{-1}(v)\cap E^1_{w}}\tilde\beta_{e^{(1)}}\tilde\gamma_{e^{(1)}}\\
=&\sum\limits_{\substack{e\in s^{-1}(v)\cap E^1_{uw},\\r(e)\not\in r(E^1_w)}}e_1e_1^*+\sum\limits_{\substack{e\in s^{-1}(v)\cap E^1_{uw},\\r(e)\in r(E^1_w)}}\sum\limits_{i=1}^{w(g^{r(e)})}e_1(g^{r(e)}_{i})^*g^{r(e)}_i(g^{r(e)}_{i})^*g^{r(e)}_ie_1^*+\sum\limits_{e\in s^{-1}(v)\cap E^1_{w}}e_1e_1^*\\
=&\sum\limits_{\substack{e\in s^{-1}(v)\cap E^1_{uw},\\r(e)\not\in r(E^1_w)}}e_1e_1^*+\sum\limits_{\substack{e\in s^{-1}(v)\cap E^1_{uw},\\r(e)\in r(E^1_w)}}\sum\limits_{i=1}^{w(g^{r(e)})}e_1(g^{r(e)}_{i})^*g^{r(e)}_ie_1^*+\sum\limits_{e\in s^{-1}(v)\cap E^1_{w}}e_1e_1^*\\
=&\sum\limits_{\substack{e\in s^{-1}(v)\cap E^1_{uw},\\r(e)\not\in r(E^1_w)}}e_1e_1^*+\sum\limits_{\substack{e\in s^{-1}(v)\cap E^1_{uw},\\r(e)\in r(E^1_w)}}e_1\big(\sum\limits_{i=1}^{w(g^{r(e)})}(g^{r(e)}_{i})^*g^{r(e)}_i\big)e_1^*+\sum\limits_{e\in s^{-1}(v)\cap E^1_{w}}e_1e_1^*\\
=&\sum\limits_{\substack{e\in s^{-1}(v)\cap E^1_{uw},\\r(e)\not\in r(E^1_w)}}e_1e_1^*+\sum\limits_{\substack{e\in s^{-1}(v)\cap E^1_{uw},\\r(e)\in r(E^1_w)}}e_1e_1^*+\sum\limits_{e\in s^{-1}(v)\cap E^1_{w}}e_1e_1^*\\
=&v=\tilde\alpha_{\tilde v}.
\end{align*}
\item Assume that $\tilde v\in N$. Then $\tilde v=v^{(i)}$ for some $v\in r(E^1_w)$ and $1\leq i\leq w(g^v)$. As mentioned above we have $\tilde s^{-1}(\tilde v)=\emptyset$ if $i=1$ and $\tilde s^{-1}(\tilde v)=\{(g^v)^{(i)}\}$ if $i>1$. Since by assumption $\tilde s^{-1}(\tilde v)\neq\emptyset$, it follows that $i>1$ and $\sum\limits_{\tilde e\in \tilde s^{-1}(\tilde v)}\tilde \beta_{\tilde e}\tilde \gamma_{\tilde e}=\tilde\beta_{(g^v)^{(i)}}\tilde\gamma_{(g^v)^{(i)}}=(g^v_i)^*g^v_i= \tilde\alpha_{\tilde v}$.
\end{enumerate}
Thus (iv) holds too and hence $\tilde X$ is an $\tilde E$-family in $L_K(E,w)$. By the Universal Property of $L_K(\tilde E)$ there is a unique $K$-algebra homomorphism $\tilde \phi: L_K(\tilde E)\rightarrow L_K(E,w)$ such that $\tilde\phi(\tilde v)=\tilde\alpha_{\tilde v}$, $\tilde\phi(\tilde e)=\tilde\beta_{\tilde e}$ and $\tilde\phi(\tilde e^*)=\tilde\gamma_{\tilde e}$ for all $\tilde v\in \tilde E^0$ and $\tilde e\in \tilde E^1$.
\\
\\
{\bf Part III} First we show that $\tilde \phi\circ\phi=\id_{L_K(E,w)}$. Clearly it suffices to show that $\tilde \phi\circ\phi$ fixes all elements of $\{v,e_i,e_i^*\mid v\in E^0, e\in E^1, 1\leq i\leq w(e)\}$ since these elements generate $L_K(E,w)$ as a $K$-algebra. One checks easily that $\tilde \phi\circ\phi$ fixes all elements $v,e_i,e_i^*$ where $v\in E^0$ and $e\in E_w^1$ or $e\in E_{uw}^1, r(e)\not\in r(E^1_w)$. Let now $e\in E_{uw}^1, r(e)\in r(E^1_w)$. Then 
\[\tilde \phi(\phi(e_1))=\tilde \phi(\sum\limits_{j=1}^{w(g^{r(e)})}e^{(j)})=\sum\limits_{j=1}^{w(g^{r(e)})}e_1(g_j^{r(e)})^*g_j^{r(e)}=e_1\sum\limits_{j=1}^{w(g^{r(e)})}(g_j^{r(e)})^*g_j^{r(e)}=e_1r(e)=e_1.\]
Similarly one can show that $\phi(\phi(e_1^*))=e_1^*$ in this case. Hence $\tilde \phi\circ\phi=\id_{L_K(E,w)}$.\\
Now we show that $\phi\circ\tilde\phi=\id_{L_K(\tilde E)}$. Clearly it suffices to show that $\phi\circ\tilde\phi$ fixes all elements of $\{\tilde v,\tilde e,\tilde e^*\mid \tilde v\in \tilde E^0, \tilde e\in \tilde E^1\}$ since these elements generate $L_K(\tilde E)$ as a $K$-algebra. One checks easily that $\phi\circ\tilde\phi$ fixes all elements $\tilde v,\tilde e,\tilde e^*$ where $\tilde v \in \tilde E^0$ and $\tilde e\in \tilde E^1\setminus B$. Let now $\tilde e\in B$. Then $\tilde e=e^{(i)}$ for some $e\in E^1_{uw}, r(e)\in r(E^1_w)$ and $1\leq i \leq w(g^{r(e)})$. Clearly
\[\phi(\tilde\phi(\tilde e))=\tilde \phi(e_1(g_i^{r(e)})^*g_i^{r(e)})=\begin{cases}\sum\limits_{j=1}^{w(g^{r(e)})}e^{(j)}((g^{r(e)})^{(1)})^*(g^{r(e)})^{(1)},&\text{ if }i=1,\\\sum\limits_{j=1}^{w(g^{r(e)})}e^{(j)}(g^{r(e)})^{(i)}((g^{r(e)})^{(i)})^*,&\text{ if }i>1.\end{cases}.\]
But $((g^{r(e)})^{(1)})^*(g^{r(e)})^{(1)}=\tilde r((g^{r(e)})^{(1)})=r(g^{r(e)})^{(1)}=r(e)^{(1)}$ in $L_K(\tilde E)$. Since $\tilde r(e^{(j)})=r(e)^{(j)}$, it follows that $\sum\limits_{j=1}^{w(g^{r(e)})}e^{(j)}((g^{r(e)})^{(1)})^*(g^{r(e)})^{(1)}=e^{(1)}=\tilde e$ if $i=1$. Now assume that $i>1$. One checks easily that $\tilde s^{-1}(r(e)^{(i)})=\{(g^{r(e)})^{(i)}\}$. Hence $(g^{r(e)})^{(i)}((g^{r(e)})^{(i)})^*=r(e)^{(i)}$ in $L_K(\tilde E)$. Since $\tilde r(e^{(j)})=r(e)^{(j)}$, it follows that $\sum\limits_{j=1}^{w(g^{r(e)})}e^{(j)}(g^{r(e)})^{(i)}((g^{r(e)})^{(i)})^*=e^{(i)}=\tilde e$. Hence we have shown that $\phi(\tilde\phi(\tilde e))=\tilde e$ if $\tilde e\in B$. Similarly one can show that $\phi(\tilde \phi(\tilde e^*))=\tilde e^*$ in this case. Hence $\phi\circ\tilde \phi=\id_{L_K(\tilde E)}$ and thus $L_K(E,w)\cong L_K(\tilde E)$.
\end{proof}

\begin{example}\label{exlpa2}
Consider the weighted graph
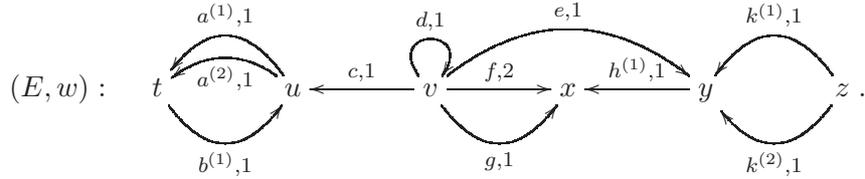
\[
(E,w):\quad\xymatrix@C+15pt{
t\ar@/_1.7pc/[r]_{b^{(1)},1}&u\ar@/_1.7pc/[l]_{a^{(1)},1}\ar@/_1.0pc/[l]^{a^{(2)},1}& v\ar[l]_{c,1}\ar@(ul,ur)^{d,1}\ar@/^1.9pc/[rr]^{e,1}\ar[r]^{f,2}\ar@/_1.7pc/[r]_{g,1}&x&y\ar[l]_{h^{(1)},1}&z\ar@/_1.7pc/[l]_{k^{(1)},1}\ar@/^1.7pc/[l]^{k^{(2)},1}
}.
\]
Let $\tilde E$ be defined as in the proof of Lemma \ref{lemkey2}. Then $\tilde E$ is the graph
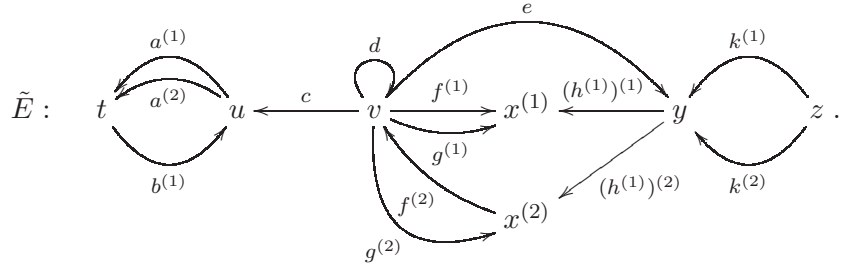
\[
\tilde E:\quad\xymatrix@C+15pt{
t\ar@/_1.7pc/[r]_{b^{(1)}}&u\ar@/_1.7pc/[l]_{a^{(1)}}\ar@/_1pc/[l]^{a^{(2)}}& v\ar[l]_{c}\ar@(ul,ur)^{d}\ar@/^2.8pc/[rr]^{e}\ar[r]^{f^{(1)}}\ar@/_0.7pc/[r]_{g^{(1)}}\ar@/_2.5pc/[dr]_{g^{(2)}}&x^{(1)}&y\ar[l]_{(h^{(1)})^{(1)}}\ar[dl]^{(h^{(1)})^{(2)}}&z\ar@/_1.7pc/[l]_{k^{(1)}}\ar@/^1.7pc/[l]^{k^{(2)}}\\
&&&x^{(2)}\ar@/^0.7pc/[ul]^{f^{(2)}}&&
}.
\]
The proof of Lemma \ref{lemkey2} shows that $L_K(E,w)\cong L_K(\tilde E)$.
\end{example}

Lemma \ref{lemkey1} and \ref{lemkey2} directly imply the theorem below.
\begin{theorem}\label{thmm}
Let $(E,w)$ be a weighted graph that satisfies Condition (LPA). Then the weighted Leavitt path algebra $L_K(E,w)$ is isomorphic to an unweighted Leavitt path algebra.
\end{theorem}

\begin{example}\label{exlpa3}
Consider the weighted graph
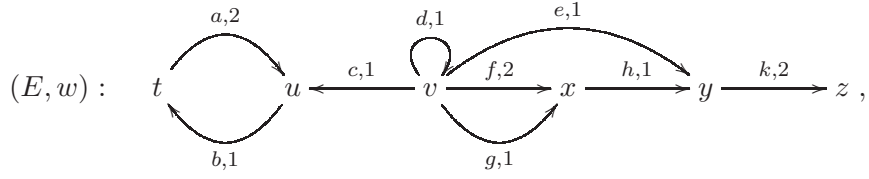
\[
(E,w):\quad\xymatrix@C+15pt{
t\ar@/^1.7pc/[r]^{a,2}&u\ar@/^1.7pc/[l]^{b,1}& v\ar[l]_{c,1}\ar@(ul,ur)^{d,1}\ar@/^1.9pc/[rr]^{e,1}\ar[r]^{f,2}\ar@/_1.7pc/[r]_{g,1}&x\ar[r]^{h,1}&y\ar[r]^{k,2}&z},
\]
which satisfies Condition (LPA), and the graph
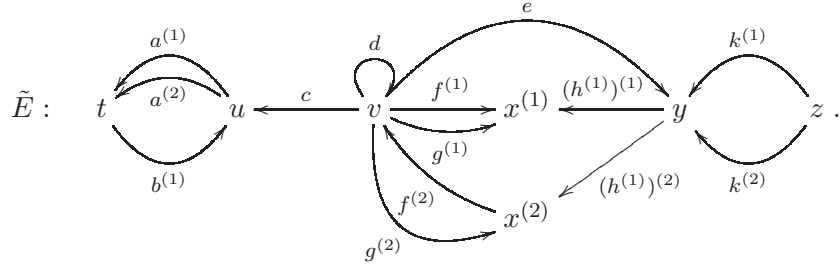
\[
\tilde E:\quad\xymatrix@C+15pt{
t\ar@/_1.7pc/[r]_{b^{(1)}}&u\ar@/_1.7pc/[l]_{a^{(1)}}\ar@/_1pc/[l]^{a^{(2)}}& v\ar[l]_{c}\ar@(ul,ur)^{d}\ar@/^2.8pc/[rr]^{e}\ar[r]^{f^{(1)}}\ar@/_0.7pc/[r]_{g^{(1)}}\ar@/_2.5pc/[dr]_{g^{(2)}}&x^{(1)}&y\ar[l]_{(h^{(1)})^{(1)}}\ar[dl]^{(h^{(1)})^{(2)}}&z\ar@/_1.7pc/[l]_{k^{(1)}}\ar@/^1.7pc/[l]^{k^{(2)}}\\
&&&x^{(2)}\ar@/^0.7pc/[ul]^{f^{(2)}}&&
}.
\]
By Examples $\ref{exlpa1}$ and \ref{exlpa2} we have $L_K(E,w)\cong L_K(\tilde E)$.
\end{example}

\section{Abscence of Condition (LPA)}
Throughout this subsection $(E,w)$ denotes a weighted graph. We start by recalling the basis result of \cite{hazrat-preusser}. Set $X:=\{v,e_i,e_i^*\mid v\in E^0,e\in E^1,1\leq i\leq w(e)\}$, let $\X$ the set of all nonempty words over $X$ and set $\overline{\X}:=\X\cup\{\text{empty word}\}$. Together with juxtaposition of words $\X$ becomes a semigroup and $\overline{\X}$ a monoid. If $A,B\in \overline{\X}$, then $B$ is called a {\it subword of $A$} if there are $C,D \in\overline{\X}$ such that $A=CBD$ and a {\it suffix of $A$} if there is a $C \in \overline{\X}$ such that $A=CB$. 

\begin{definition}
Let $p=x_1\dots x_n\in\X$. Then $p$ is called {\it a d-path} if either $x_1,\dots,x_n\in X\setminus E^0$ and $r(x_i)=s(x_{i+1})~(1\leq i \leq n-1)$ or $x_1\in E^0$ and $n=1$. Here we use the convention $s(v):=v$, $r(v):=v$, $s(e_i):=s(e)$, $r(e_i):=r(e)$, $s(e^*_i):=r(e)$ and $r(e^*_i):=s(e)$ for any $v\in E^0$, $e \in E^1$ and $1\leq i \leq w(e)$.
\end{definition} 

\begin{remark}
Let $\hat E$ be the directed graph associated to $(E,w)$ and $\hat E_d$ the double graph of $\hat E$ (see \cite[Definitions 2 and 8]{preusser}). The d-paths are precisely the paths in the double graph $\hat E_d$.
\end{remark}

Fix for any $v\in E^0$ such that $s^{-1}(v)\neq\emptyset$ an edge $e^v\in s^{-1}(v)$ such that $w(e^v)=w(v)$. The $e^v$'s are called {\it special edges}.

\begin{definition}
The words $e^v_i(e^v_j)^*~(v\in E^0,1\leq i,j\leq w(v))$ and $e^*_1f_1~(e,f\in E^1)$
in $\X$ are called {\it forbidden}. A {\it normal d-path} or {\it nod-path} is a d-path $p$ such that none of its subwords is forbidden.
\end{definition}

Let $K\X$ the free $K$-algebra generated by $X$ (i.e. the $K$-vector space with basis $\X$ which becomes a $K$-algebra by linearly extending the juxtaposition of words). Then $L_K(E,w)$ is the quotient of $K\X$ by the ideal generated by the relations (i)-(iv) in Definition \ref{def3}. Let $K\X_{\nod}$ be the linear subspace of $K\X$ spanned by the nod-paths.

\begin{theorem}[Hazrat, Preusser, 2017] \label{thmbasis}
The canonical map $K\X_{\nod}\rightarrow L_K(E,w)$ is an isomorphism of $K$-vector spaces. In particular the images of the nod-paths under this map form a linear basis for $L_K(E,w)$.
\end{theorem} 
\begin{proof}
See \cite[Theorem 16]{hazrat-preusser} and its proof.
\end{proof}



The following lemma will be used in the proofs of Theorems \ref{thm1},\ref{thm2},\ref{thm3},\ref{thm4},\ref{thm5},\ref{thm6} and \ref{thmm2}.
\begin{keylemma}\label{lemimp}
Suppose that $(E,w)$ does not satisfy Condition (LPA). Then there is a nod-path whose first letter is $e_2$ and whose last letter is $e_2^*$ for some $e\in E^1_w$.
\end{keylemma}
\begin{proof}
\cite[Proof of Lemma 35]{preusser} shows that if one of the Conditions (LPA1), (LPA2) and (LPA3) is not satisfied, then then there is a nod-path whose first letter is $e_2$ and whose last letter is $e_2^*$ for some $e\in E^1_w$. Assume now that $(E,w)$ does not satisfy Condition (LPA4). Then there is an $e\in E^1_w$, a path $p$ and a cycle $c$ such that $s(p)=r(e)$, $r(p)=s(c)$ and $e$ does not belong to $c$. Write $c=f^{(1)}\dots f^{(m)}$ where $f^{(1)},\dots,f^{(m)}\in E^1$. If $p=r(e)$, then $e_2f^{(1)}_1\dots f^{(m)}_1e_2^*$ is a nod-path (since $f^{(m)}\neq e$). Now assume that $p=g^{(1)}\dots g^{(n)}$ where $g^{(1)},\dots,g^{(n)}\in E^1$. Clearly we assume that no letter of $p$ is a letter of $c$. One checks easily that $e_2g^{(1)}_1\dots g^{(n)}_1f^{(1)}_1\dots f^{(m)}_1(g^{(n)}_1)^*\dots(g^{(1)}_1)^*e_2^*$ is a nod-path (note that $f^{(m)}\neq g^{(n)}$).
\end{proof}

\begin{theorem}\label{thm1}
Suppose that $(E,w)$ does not satisfy Condition (LPA). Then $L_K(E,w)$ is neither simple nor graded simple.
\end{theorem}
\begin{proof}
By Lemma \ref{lemimp}, there is a nod-path $p$ whose first letter is $e_2$ and whose last letter is $e_2^*$ for some $e\in E^1_w$. One checks easily that the ideal $I$ generated by $p$ equals the linear span of all nod-paths that contain $p$ as a subword (note that $e_2$ is not the second letter of a forbidden word and $e_2^*$ not the first letter of a forbidden word). It follows that $I$ is a proper ideal of $L_K(E,w)$ (it is not the zero ideal since it contains the basis element $p$ and it is not equal to $L_K(E,w)$ since it does not contain any vertex). Since $I$ is generated by a homogeneous element, it is a graded ideal. 
\end{proof}

Recall that a group graded $K$-algebra $A=\bigoplus\limits_{g\in G} A_g$ is called {\it locally finite} if $\dim_K A_g < \infty$ for every $g\in G$. 
\begin{theorem}\label{thm2}
Suppose that $(E,w)$ does not satisfy Condition (LPA). Then $L_K(E,w)$ is not locally finite.
\end{theorem}
\begin{proof}
By Lemma \ref{lemimp}, there is a nod-path $p=x_1\dots x_n$ such that $x_1=e_2$ and $x_n=e_2^*$ for some $e\in E^1_w$. Set $p^*:=x_n^*\dots x_1^*$ (where $(f_i^*)^*=f_i$ for any $f\in E^1$ and $1\leq i\leq w(f)$). One checks easily that for any $n\in\N$, $(pp^*)^n$ is a nod-path that lies in the homogeneous $0$-component $L_K(E,w)_0$. It follows from Theorem \ref{thmbasis} that $\dim_K(L_K(E,w)_0)=\infty$.
\end{proof}

\begin{theorem}\label{thm3}
Suppose that $(E,w)$ does not satisfy Condition (LPA). Then $L_K(E,w)$ is not Noetherian.
\end{theorem}
\begin{proof}
By Lemma \ref{lemimp}, there is a nod-path $p$ whose first letter is $e_2$ and whose last letter is $e_2^*$ for some $e\in E^1_w$. Let $q$ be the nod-path one gets by replacing the first letter of $p$ by $e_1$. For any $n\in\N$ let $I_n$ be the left ideal generated by the nod-paths $p,pq,\dots,pq^n$. One checks easily that $I_n$ equals the linear span of all nod-paths $o$ such that one of the words $p,pq,\dots,pq^n$ is a suffix of $o$. It follows that 
$I_n\subsetneq I_{n+1}$ (clearly none of the words $p,pq,\dots,pq^n$ is a suffix of $pq^{n+1}$ since $p$ and $q$ have the same length but are distinct; hence $pq^{n+1}\not\in I_n$). 
\end{proof}

\begin{theorem}\label{thm4}
Suppose that $(E,w)$ does not satisfy Condition (LPA). Then $L_K(E,w)$ is not Artinian.
\end{theorem}
\begin{proof}
By Lemma \ref{lemimp}, there is a nod-path $p$ whose first letter is $e_2$ and whose last letter is $e_2^*$ for some $e\in E^1_w$. For any $n\in\N$ let $I_n$ be the left ideal generated by $p^n$. One checks easily that $I_n$ equals the linear span of all nod-paths $o$ such that $p^n$ is a suffix of $o$. Hence $I_n\supsetneq I_{n+1}$ (clearly $p^{n+1}$ is not a suffix of $p^n$ and hence $p^n\not\in I_{n+1}$).
\end{proof}

\begin{theorem}\label{thm5}
Suppose that $(E,w)$ does not satisfy Condition (LPA). Then $L_K(E,w)$ is not von Neumann regular.
\end{theorem}
\begin{proof}
By Lemma \ref{lemimp}, there is a nod-path $p$ whose first letter is $e_2$ and whose last letter is $e_2^*$ for some $e\in E^1_w$. One checks easily that for any $x\in L_K(E,w)$, $pxp$ is a linear combination of nod-paths of length $\geq 2|p|$. Hence the equation $pxp=p$ has no solution $x\in L_K(E,w)$.
\end{proof}

We recall some general facts on the growth of algebras. Let $A\neq\{0\}$ be a finitely generated $K$-algebra. Let $V$ be a {\it finite-dimensional generating subspace} of $A$, i.e. a finite-dimensional subspace of $A$ that generates $A$ as a $K$-algebra. For $n\geq 1$ let $V^n$ denote the linear span of the set $\{v_1\dots v_k\mid k\leq n, v_1,\dots,v_k\in V\}$. Then 
\[V =V^1\subseteq V^2\subseteq V^3\subseteq \dots, \quad A =\bigcup\limits_{n\in \N}V^n\text{ and }d_V(n):=\dim V^n<\infty.\] 
Given functions $f, g:\N\rightarrow \R^+$, we write $f\preccurlyeq g$ if there is a $c\in\N$ such that $f(n)\leq cg(cn)$ for all $n$. If $f\preccurlyeq g$ and $g\preccurlyeq f$, then the functions $f, g$ are called {\it asymptotically equivalent} and we write $f\sim g$. If $W$ is another finite-dimensional generating subspace of $A$, then $d_V\sim d_W$. The {\it Gelfand-Kirillov dimension} or {\it GK dimension} of $A$ is defined as
\[\GKdim A := \limsup\limits_{n\rightarrow \infty}\log_nd_V(n).\]
The definition of the GK dimension does not depend on the choice of the finite-dimensional generating subspace $V$. If $d_V\preccurlyeq n^m$ for some $m\in \N$, then $A$ is said to have {\it polynomial growth} and we have $\GKdim A \leq m$. If $d_V\sim a^n$ for some real number $a>1$, then $A$ is said to have {\it exponential growth} and we have $\GKdim A =\infty$. If $A$ does not happen to be finitely generated over $K$, then the GK dimension of $A$ is defined as
\[\GKdim(A) := \sup\{\GKdim(B)\mid B \text{ is a finitely generated subalgebra of }A\}.\]
For the algebra $A=\{0\}$ we set $\GKdim A:=0$.

\begin{theorem}\label{thm6}
Suppose that $(E,w)$ does not satisfy Condition (LPA). Then $\GKdim(L_K(E,w))=\infty$.
\end{theorem}
\begin{proof}
Suppose first that $(E,w)$ is finite (in our setting that means that $E^0$ is a finite set). By Lemma \ref{lemimp}, there is a nod-path $p$ in $( E, w)$ whose first letter is $e_2$ and whose last letter is $e_2^*$ for some $e\in  E^1_w$. Let $q$ be the nod-path one gets by replacing the first letter of $p$ by $e_1$. Let $n\in \N$. Consider the nod-paths
\begin{equation}
p^{i_1}q^{i_2}\dots p^{i_{k-1}}q^{i_{k}}~ (k\text{ even}), \text{ and }p^{i_1}q^{i_2}\dots p^{i_{k-2}}q^{i_{k-1}}p^{i_{k}}~(k\text{ odd})
\end{equation}
where $k,i_1,\dots,i_k\in \N$ satisfy
\begin{equation}
(i_1+\dots+i_k)|p|\leq n.
\end{equation}
Clearly different solutions $(k,i_1,\dots,i_k)$ and $(k',i'_1,\dots,i'_{k'})$ of inequality (2) correspond to different nod-paths in (1) since $|p|=|q|$. Let $V$ denote the finite-dimensional subspace of $L_K( E, w)$ spanned by $\{v,f_i,f_i^*\mid v\in  E^0, f\in  E_1, 1\leq i\leq  w(f)\}$. By Theorem \ref{thmbasis} the nod-paths in (1) are linearly independent in $V^n$. The number of solutions of (2) is $\sim 2^n$ and hence $L_K( E, w)$ has exponential growth.\\
Now suppose that $(E,w)$ is not finite. One checks easily that there is a finite complete weighted subgraph $(\tilde E, \tilde w)$ of $(E,w)$ that does not satisfy Condition (LPA) (see \cite[p. 884 and Proof of Lemma 5.19]{hazrat13}). By the previous paragraph $L_K( \tilde E, \tilde w)$ has exponential growth. Clearly the inclusion $(\tilde E,\tilde w)\hookrightarrow (E,w)$ induces an algebra monomorphism $L_K(\tilde E,\tilde w)\rightarrow L_K(E,w)$ since one can choose the special edges such that distinct nod-paths are mapped to distinct nod-paths. Hence $L_K(E,w)$ has a finitely generated subalgebra with exponential growth. It follows from the definition of the GK dimension that $\GKdim L_K(E,w)=\infty$.
\end{proof}

The main result of this section is Theorem \ref{thmm2}. In order to prove it we need two lemmas.

\begin{lemma}\label{lemidem}
Let $p$ be a nod-path starting with $e_2$ and ending with $e_2^*$ for some $e\in E^1_w$. Then the ideal $I$ of $L_K(E,w)$ generated by $p$ contains no nonzero idempotent.
\end{lemma}
\begin{proof}
For a nod-path $q=x_1\dots x_n$ define $m(q)$ as the largest nonnegative integer $m$ such that there are indices $i_1,\dots,i_m\in \{1,\dots,n\}$ such that $i_j+|p|-1<i_{j+1}~(1\leq j \leq m-1)$, $i_m+|p|-1\leq n$ and $x_{i_j}\dots x_{i_j+|p|-1}=p~(1\leq j \leq m)$. Hence $m(q)$ is maximal with the property that $q$ contains $m(q)$ not overlapping copies of $p$.\\
Now let $a\in I\setminus \{0\}$. By Theorem \ref{thmbasis} we can write $a=\sum \limits_{r=1}^{t}k_rq_r$ where $k_1,\dots,k_t\in K\setminus\{0\}$ and $q_1,\dots,q_t$ are pairwise distinct nod-paths. Clearly $m(q_r)\geq 1$ for any $1\leq r\leq t$, since $I$ consists of all linear combinations of nod-paths containing $p$ as a subword. It easy to show, using the fact that $e_2$ is not the second letter of a forbidden word and $e_2^*$ not the first letter of a forbidden word, that for any $1\leq r,s\leq t$ the product $q_rq_s$ is a linear combination of nod-paths $o$ such that $m(o)\geq m(q_r)+m(q_s)$ (cf. \cite[Proof of Proposition 40]{hazrat-preusser}). It follows that $a^2=\sum \limits_{r,s=1}^{t} k_rk_sq_rq_s$ is a linear combination of nod-paths $o$ such that $m(o)\geq 2m(q_{r_{\min}})>m(q_{r_{\min}})$ where $1\leq r_{\min}\leq t$ is chosen such that $m(q_{r_{\min}})$ is minimal. Hence $a^2$ is a linear combination of nod-paths none of which equals $q_{r_{\min}}$. Thus $a^2$ cannot be equal to $a$.
\end{proof}

If $\Lambda$ is an infinite set and $S$ is a unital ring, then we denote by $M_{\Lambda}(S)$ the $K$-algebra consisting of all square matrices $M$, with rows and columns indexed by $\Lambda$, with entries from $S$, for which there are at most finitely many nonzero entries in $M$ (cf. \cite[Notation 2.6.3]{abrams-ara-molina}).

\begin{lemma}\label{lemmorita}
Let $\Lambda$ be an infinite set and $S$ a left Noetherian, unital ring. Let $I_1\subseteq I_2\subseteq\dots$ be an ascending chain of left ideals of $M_{\Lambda}(S)$. Suppose there is a finite subset $\Lambda^{\fin}$ of $\Lambda$ such that $\sigma_{\lambda\mu}=0$ for any $n\in\N$, $\sigma\in I_n$, $\lambda\in\Lambda$ and $\mu\in  \Lambda\setminus \Lambda^{\fin}$. Then the chain $I_1\subseteq I_2\subseteq\dots$ eventually stabilises.
\end{lemma}
\begin{proof}
Write $\Lambda^{\fin}=\{\lambda_1,\dots,\lambda_m\}$. Fix a $\tau\in \Lambda$. For any $n\in\N$, let $N_n$ be the left $S$-submodule of $S^m$ consisting of all row vectors $(\sigma_{\tau\lambda_1},\dots,\sigma_{\tau\lambda_m})$ where $\sigma$ varies over all matrices in $I_{n}$. Then $I_n$ equals the set of all matrices $\sigma\in M_{\Lambda}(S)$ such that $\sigma_{\lambda\mu}=0$ for any $\lambda\in \Lambda,\mu\in \Lambda\setminus \Lambda^{\fin}$ and $(\sigma_{\lambda\lambda_1},\dots,\sigma_{\lambda\lambda_m})\in N_n$ for any $\lambda\in \Lambda$. Since $S$ is a left Noetherian ring, $S^m$ is a Noetherian module. It follows that the chain $N_1\subseteq N_2\subseteq \dots$ eventually stabilises and thus the chain $I_1\subseteq I_2\subseteq \dots$ eventually stabilises.
\end{proof}

\begin{theorem}\label{thmm2} Suppose that $(E,w)$ does not satisfy Condition (LPA). Then $L_K(E,w)$ is not isomorphic to an unweighted Leavitt path algebra.
\end{theorem}
\begin{proof}
Assume there is a graph $F$ and an isomorphism $\phi:L_K(E,w)\rightarrow L_K(F)$. By Lemma \ref{lemimp}, there is a nod-path $p$ whose first letter is $e_2$ and whose last letter is $e_2^*$ for some $e\in E^1_w$. Let $q$ be the nod-path one gets by replacing the last letter of $p$ by $e_1^*$. By Lemma \ref{lemidem}, the ideal $I$ of $L_K(E,w)$ generated by $p$ contains no nonzero idempotent. Similarly, for any $n\in \N$, the ideal $I_n$ of $L_K(E,w)$ generated by $qp^n$ contains no nonzero idempotent. It follows from \cite[Proposition 2.7.9]{abrams-ara-molina}, that $\phi(I),\phi(I_n)\subseteq I(P_c(F))~(n\in\N)$ where $I(P_c(F))$ is the ideal of $L_K(F)$ generated by all vertices in $F^0$ which belong to a cycle without an exit. It follows that $\phi(p),\phi(qp^n)\in I(P_c(F))~(n\in\N)$. By \cite[Theorem 2.7.3]{abrams-ara-molina} we have 
\begin{equation}
I(P_c(F))\cong \bigoplus\limits_{i\in \Gamma}M_{\Lambda_i}(K[x,x^{-1}])
\end{equation}
as a $K$-algebra. The sets $\Gamma$ and $\Lambda_i~(i\in \Gamma)$ in (3) might be infinite if $F$ is not finite. \\
It follows from the previous paragraph that there is a subalgebra $A$ of $L_K(E,w)$ such that $p,qp^n\in A~(n\in \N)$ and $A\cong \bigoplus\limits_{i\in \Gamma}M_{\Lambda_i}(K[x,x^{-1}])$. For any $n\in \N$ let $J_n$ be the left ideal of $A$ generated by $qp^2,\dots,qp^{n+1}$. Then $J_n$ is contained in the linear span of all nod-paths $o$ such that one of the words $qp^2,\dots,qp^{n+1}$ is a suffix of $o$. It follows that 
$J_n\subsetneq J_{n+1}$ (clearly none of the words $qp^2,\dots,qp^{n+1}$ is a suffix of $qp^{n+2}$ since $p$ and $q$ have the same length but are distinct). If the sets $\Gamma$ and $\Lambda_i~(i\in \Gamma)$ are finite, then we already have a contradiction since it is well-known that $\bigoplus\limits_{i\in \Gamma}M_{\Lambda_i}(K[x,x^{-1}])$ is Noetherian in this case. Hence the next two paragraphs are only needed if one of the sets $\Gamma$ and $\Lambda_i~(i\in \Gamma)$ is infinite.\\
If $a\in A$, then we identify $a$ with its image in $\bigoplus\limits_{i\in \Gamma}M_{\Lambda_i}(K[x,x^{-1}])$ and write $a_i$ for the $i$-th component of $a$. Set $\Gamma^{\fin}:=\{i\in \Gamma\mid p_i\neq 0\}$. Then $\Gamma^{\fin}$ is a finite subset of $\Gamma$. Clearly $(qp^n)_i=0$ for any $i\in \Gamma\setminus\Gamma^{\fin}$ and $n\geq 2$ (since $(qp^n)_i=(qp^{n-1}p)_i=(qp^{n-1})_ip_i$ for any $n\geq 2$). Hence we can reduce to the case that $\Gamma$ is finite.\\
For any $n\in \N$ and $i\in \Gamma$, let $J_{n,i}$ be the left ideal of $M_{\Lambda_i}(K[x,x^{-1}])$ generated by $(qp^2)_i,\dots,(qp^{n+1})_i$. Then $J_n=\bigtimes\limits_{i\in\Gamma}J_{n,i}$ since each $M_{\Lambda_i}(K[x,x^{-1}])$ has local units. Now fix an $i\in \Gamma$. Let $\Lambda_i^{\fin}$ be the finite subset of $\Lambda_i$ consisting of all $\lambda\in \Lambda_i$ such that the $\lambda$-th column of $p_i$ has a nonzero entry. Then clearly $\sigma_{\lambda\mu}=0$ for any $n\in \N$, $\sigma\in J_{n,i}$, $\lambda\in\Lambda_i$ and $\mu\in\Lambda_i\setminus \Lambda_i^{\fin}$ (since any element of $J_{n,i}$ is a left multiple of $p_i$). Hence, by Lemma \ref{lemmorita}, the chain $J_{1,i}\subseteq J_{2,i}\subseteq \dots$ eventually stabilises. Since this holds for any $i\in \Gamma$, we get the contradiction that the chain $J_{1}\subseteq J_{2}\subseteq \dots$ eventually stabilises. 
\end{proof}

\section{Summary}

\begin{theorem}
Let $(E,w)$ be a row-finite weighted graph and $K$ a field. Then $L_K(E,w)$ is isomorphic to an unweighted Leavitt path algebra iff $(E,w)$ satisfies Condition (LPA) (see Definition \ref{defLPA}).
\end{theorem}
\begin{proof}
Follows from the Theorems \ref{thmm} and \ref{thmm2}.
\end{proof}

\begin{theorem}
Let $(E,w)$ be a row-finite weighted graph and $K$ a field. If $L_K(E,w)$ is simple, or graded simple, or locally finite, or Noetherian, or Artinian, or von Neumann regular, or has finite GK dimension, then $L_K(E,w)$ is isomorphic to an unweighted Leavitt path algebra.
\end{theorem}
\begin{proof}
Follows from the Theorems \ref{thm1},\ref{thm2},\ref{thm3},\ref{thm4},\ref{thm5},\ref{thm6} and \ref{thmm}.
\end{proof}

\end{document}